\documentclass[11pt,draft]{amsart}
\usepackage{amsmath,amssymb,amsthm}
\usepackage[latin1]{inputenc}
\usepackage{version,tabularx,multicol}
\usepackage{graphicx,float,psfrag}
\usepackage{stmaryrd}

\newcommand{\cT}{{\mathfrak T}}

\newcommand{\reff}[1]{(\ref{#1})}

\theoremstyle{plain}
\newtheorem{theo}{Theorem}[section]
 
\newtheorem{cor}[theo]{Corollary}
 
\newtheorem{prop}[theo]{Proposition}
 \newtheorem{proposition}[theo]{Proposition}
\newtheorem{lem}[theo]{Lemma}
 \newtheorem{lemma}[theo]{Lemma}

\theoremstyle{remark}
\newtheorem{rem}[theo]{Remark}
 \newtheorem{remark}[theo]{Remark}
 \newtheorem{exple}[theo]{Example}

\newcommand{\cb}{{\mathcal B}}

\newcommand{\cf}{{\mathcal F}}
\newcommand{\cg}{{\mathcal G}}
\newcommand{\ch}{{\mathcal H}}

\newcommand{\ck}{{\mathcal K}}

\newcommand{\cn}{{\mathcal N}}
\newcommand{\cm}{{\mathcal M}}
\newcommand{\cp}{{\mathcal P}}
\newcommand{\cq}{{\mathcal Q}}
\newcommand{\cs}{{\mathcal S}}

\newcommand{\ct}{{\mathcal T}}
\newcommand{\cx}{{\mathcal X}}

\newcommand{\cz}{{\mathcal Z}}

\newcommand{\DD}{{\mathbb D}}
\newcommand{\E}{{\mathbb E}}

\newcommand{\F}{{\mathbb F}}

\newcommand{\N}{{\mathbb N}}
\renewcommand{\P}{{\mathbb P}}

\newcommand{\R}{{\mathbb R}}

\newcommand{\T}{{\mathbb T}}
\newcommand{\Z}{{\mathbb Z}}

\newcommand{\rP}{{\rm P}}
\newcommand{\rE}{{\rm E}}

\newcommand{\bm}{\mathbf m}
\newcommand{\bN}{{\mathbf N}}

\newcommand{\ind}{{\bf 1}}

\newcommand{\Br}{{\rm Br}}
\newcommand{\Lf}{{\rm Lf}}
\newcommand{\Sk}{{\rm Sk}}

\newcommand{\Card}{{\rm Card}\;}

\newcommand{\inv}[1]{\mathop{\frac{1}{ #1}}\nolimits}
\newcommand{\expp}[1]{\mathop {\mathrm{e}^{ #1}}}

 \def\beqlb{\begin{eqnarray}}\def\eeqlb{\end{eqnarray}}
 \def\beqnn{\begin{eqnarray*}}\def\eeqnn{\end{eqnarray*}}

 \def\mbb{\mathbb}

\newcommand{\bcen}{\begin{center}}
\newcommand{\ecen}{\end{center}}
\newcommand{\bgeqn}{\begin{equation}}
\newcommand{\edeqn}{\end{equation}}

\def\dz{\delta}
\def\lz{\lambda}

\def\F{{\mathcal F}}
\def\G{{\mathcal G}}

\def\P{{\mathcal  P}}

\def\cn{{\mathcal  N}}

\def\ct{{\mathcal  T}}
\def\Z{{\mathcal  Z}}
\def\mN{{\mbb  N}}
\def\N{{\mbb  N}}

\def\E{{\mbb  E}}

\def\P{{\mbb  P}}

\def\R{{\mbb  R}}

\def\rar{\rightarrow}

\def\llb{\llbracket}
\def\rrb{\rrbracket}

\newcommand{\lb}{\llbracket}
\newcommand{\rb}{\rrbracket}

\begin{document}

\title[CRT-sub-trees]{Pruning of CRT-sub-trees}
\date{\today}

\author{Romain Abraham}

\address{
Romain Abraham,
MAPMO, CNRS UMR 6628,
F\'ed\'eration Denis Poisson FR 2964,
Universit\'e d'Orl\'eans,
B.P. 6759,
45067 Orl\'eans cedex 2
FRANCE.
}

\email{romain.abraham@univ-orleans.fr}

\author{Jean-Fran\c cois Delmas}

\address{
Jean-Fran\c cois Delmas,
CERMICS,  Universit\'e Paris-Est, 6-8
av. Blaise Pascal,
  Champs-sur-Marne, 77455 Marne La Vall\'ee, FRANCE.}

\email{delmas@cermics.enpc.fr}

\author{Hui He}

\address{
Hui He, Laboratory of Mathematics and Complex Systems, School of Mathematical
Sciences, Beijing Normal University, Beijing 100875, P.R.CHINA. }

\email{hehui@bnu.edu.cn}

\thanks{This work is partially supported by the ``Agence Nationale de la
  Recherche'',  ANR-BLAN06-3-146282  and ANR-08-BLAN-0190, and by NSFC
  (No. 11126037, 11201030).}

\begin{abstract}
We  study  the  pruning  process  developed  by Abraham and Delmas (2012)
on the discrete Galton-Watson sub-trees of the L\'{e}vy tree which are  obtained by
considering the  minimal sub-tree connecting  the root and  leaves chosen
uniformly at rate $\lambda$,  see  Duquesne  and  Le  Gall (2002). 
The  tree-valued process, as  $\lambda$ increases, has
been  studied by  Duquesne and  Winkel (2007). 
  Notice that we have  a tree-valued  process indexed  by two
parameters the pruning parameter $\theta$  and the intensity $\lambda$.
Our  main  results are:  construction and
marginals   of   the   pruning   process, 
representation  of the  pruning  process  (forward in  time  that is  as
$\theta$ increases) and description of the growing process (backward in time
that  is as  $\theta$  decreases) and distribution of the ascension time
(or explosion time of the backward process) as well as  the tree
at the ascension time. A by-product of our result is that the
super-critical L\'{e}vy trees independently introduced by Abraham and Delmas (2012)
and Duquesne and  Winkel (2007) coincide. This work is also related to
the pruning of discrete Galton-Watson trees studied by Abraham, Delmas
and He (2012).
\end{abstract}



\maketitle

\section{Introduction}

The study of pruning of Galton-Watson trees has been initiated by Aldous
and Pitman  \cite{ap:tvmcdgwp}. Roughly  speaking, it corresponds  to the
percolation on edges:  an edge is uniformly chosen at random in the Galton-Watson
tree and it  is removed and only the  connected component containing the
root remains.   This procedure  is then iterated.   This process  can be
extended  backward in  time.  It  corresponds then  to  a non-decreasing
tree-valued process. The  ascension time $A$, is then  the first time at
which   this  tree-valued   process   reaches  an   unbounded  tree.    In
\cite{ap:tvmcdgwp}, the  authors give the  joint distribution of  $A$ as
well as the tree just before the ascension time (in backward time).  The
limits of Galton-Watson trees are  the so called continuum L\'{e}vy trees, see
\cite{a:crt1,a:crt3,lglj:bplpep,dlg:pfalt};  they are characterized  by a
branching mechanism $\psi$ which is also a L\'{e}vy exponent. The result for
the pruning process on Galton-Watson  trees was then extended by Abraham
and  Delmas  \cite{ad:ctvmp} to  a process indexed by time $\theta$
whose marginals are continuum L\'{e}vy  trees. In  the
setting of  the Brownian continuum  random tree, which corresponds  to a
quadratic branching  mechanism, the pruning procedure is  uniform on the
skeleton, see  also Aldous and Pitman \cite{ap:sac}  for a fragmentation
point   of   view   in   this   case.    This   is   the   analogue   of
\cite{ap:tvmcdgwp}. However  in the general  L\'{e}vy case, one has  to take
into account the  pruning of nodes with a rate given  by its ``size'' or
``mass'',  which is  defined as  the  asymptotic number  of small  trees
attached to the node.  This result in the continuous setting motivated a
new pruning  procedure on  the nodes of  Galton-Watson trees,  which was
developed by Abraham, Delmas  and He \cite{adh:pgwttvmp}.  In this case,
the pruning  happens on the nodes  with rate depending on  the degree of
the nodes.

In  the  present  work,  we  study  the  pruning  process  developed  in
\cite{ad:ctvmp} on the discrete Galton-Watson sub-tree of the L\'{e}vy tree.
The discrete  Galton-Watson sub-trees of  the L\'{e}vy trees are  obtained by
considering the  minimal sub-tree connecting  the root and  leaves chosen
uniformly  with  rate  $\lambda\geq   0$,  see  Duquesne  and  Le  Gall
\cite{dlg:pfalt}. The  tree-valued process, as  $\lambda$ increases, has
been  studied by  Duquesne and  Winkel \cite{dw:glt},  in  particular to
construct  super-critical L\'{e}vy  trees. Notice  that  super-critical L\'{e}vy
trees have also  been defined in \cite{ad:ctvmp}. One  of the by-product
of  our  results is  that  the  two  definitions coincide,  see  Section
\ref{SecsubCov}.   Notice that we have  a tree-valued  process indexed  by two
parameters $\theta$  (as in \cite{ap:tvmcdgwp,  ad:ctvmp}) and $\lambda$
(as  in \cite{dw:glt}).  The  other main  results are:  construction and
marginals   of   the   pruning   process  in   Section   \ref{SecsubMG},
representation  of the  pruning  process  (forward in  time  that is  as
$\theta$ increases) and description of the growing process (backward in time
that  is as  $\theta$  decreases) in  Section \ref{sec:subgrowth},  some
remarks  on martingales  related  to  the number  of  leaves in  Section
\ref{sec:leaves}, distributions of the ascension time and of the tree
at the ascension time in Section \ref{sec:ascension}.

Now, we present more precisely our results. Let $\psi$ be a branching
mechanism satisfying some regularity conditions (see (H1-3) in Section
\ref{sec:exc-meas}). We define $\psi_\theta$ by:
\[
\psi_\theta(q)=\psi(q+\theta) -\psi(\theta) \quad \text{for all $q\geq 0$},
\]
and set  $\Theta^\psi$ the  set of $\theta$  for which  $\psi_\theta$ is
well  defined.   We   consider  the  tree-valued  process  $(\ct_\theta,
\theta\in \Theta^\psi)$ introduced  in \cite{ad:ctvmp}, corresponding to
a uniform  pruning on the skeleton and  to a pruning at  nodes with rate
depending on its  size. We recall that $\ct_\theta$ is  a L\'{e}vy tree with
branching  mechanism $\psi_\theta$. Let  $\bm^{\ct_\theta}$ be  its mass
measure, which is a uniform measure on the set of leaves. Let $\tau_0(\lambda)$
be  the minimal sub-tree  of $\ct_0$  generated by  the root  and leaves
chosen  before  time $\lambda$  according  to  a  Poisson point  measure
$\cp^0$ on  $\R_+\times \ct_0$ with intensity  $dt \, \bm^{\ct_\theta}$.
Let     $M_\lambda$    be     the    number     of     chosen    leaves:
$M_\lambda=\cp^0([0,\lambda]\times \ct_0)$, so that $\tau_0(\lambda)$ is
well     defined    for     $M_\lambda\geq    1$.      And     we    set
$\tau_\theta(\lambda)=\ct_\theta     \bigcap     \tau_0(\lambda)$    for
$\theta\geq  0$.    So  we  get  a  two-parameter   family  of  sub-trees
$(\ct_\theta(\lambda),  \lambda\geq 0,  \theta\geq  0)$.  Let  $\P^{\psi,
  \lambda}$   be   the    conditional probability  given   the   event
$\{M_\lambda\geq 1\}$.  We will be interested in the process
$\lambda\mapsto \tau_\theta(\lambda)$ which was studied in
\cite{dw:glt} and in the pruning process $\theta\mapsto
\tau_\theta(\lambda)$, which for $\lambda=+\infty $ was studied in
\cite{ad:ctvmp}.

Notice  that the leaves of $\tau_\theta(\lambda)$
correspond to marked  leaves belonging to $\ct_\theta$ as  well as roots
of  sub-trees of $\ct_0$  with marked  leaves which  are removed  to get
$\ct_\theta$. If one is interested only in $\hat \tau_\theta(\lambda)$,
the minimal sub-tree containing the root and the marked leaves
belonging to $\ct_\theta$, then one would get a process such that
$\hat \tau_\theta(\lambda)$ has the same distribution as
$\tau_\theta(\lambda_\theta)$ with $\lambda_\theta=\psi(\psi_\theta^{-1}
(\lambda)))$. This would lead to another natural process index by the
level-set of the function $(\theta,\lambda)\mapsto
\psi_\theta(\psi^{-1}(\lambda))$.

Theorem  3.2.1  in  \cite{dlg:rtlpsbp}  in  the  sub-critical  case  and
Corollary \ref{cor:tau-law} in this paper  in the general case gives  that the sub-tree
$\tau_{\theta}(\lz)$ is distributed as a Galton-Watson real tree; its
reproduction    law    has    generating   function    $g_{(\psi_\theta,
  \psi_\theta(\eta))}$,    see     definition    \reff{eq:def-g},    with
$\eta=\psi^{-1}(\lambda)$  and   exponential  individual  lifetime  with
parameter $ \psi_\theta'(\eta)$.  If we endow $\ct_\theta$ with its mass
measure and $\tau_\theta(\lambda)$ with  a discrete mass measure defined
by
\[
\bm^{\tau_\theta(\lambda)}=\inv{\psi_\theta(\eta)} \sum_{\text{$x$ a
    leaf of $\tau_\theta(\lambda)$}} \delta_x,
\]
then we have in  Theorem \ref{theo:cv-dGHPc} the
convergence  for  the   Gromov-Hausdorff-Prohorov  distance  defined  in
\cite{adh:ghptwms} of $\tau_{\theta}(\lz)$  to $\ct_\theta$ as $\lambda$
goes to infinity. This result  was already in \cite{dw:glt} (with the Gromov-Hausdorff distance instead of the Gromov-Hausdorff-Prohorov distance),  and this  insures  that  in  the
super-critical  case the  L\'{e}vy trees  introduced in  \cite{dw:glt}  and in
\cite{ad:ctvmp} are the same. We give in Theorems \ref{theo:pruning} and
\ref{theo:growth}    a    precise    description    of    the    process
$(\tau_\theta(\lambda),   \theta\geq   0)$   in   forward   (decreasing
tree-valued process) and backward (increasing tree-valued process) time.
By considering the backward process, we see it is possible to extend the
process up to $\theta_\lambda$ backward in time, with $\theta_\lambda$
defined roughly by $\psi_\theta(\psi^{-1}(\lambda))=0$ (see \reff{eq:def-q-l} for a
precise definition). Usually $\theta_\lambda$ is not the lower bound of
$\Theta^\psi$. Intuitively, when $\theta$ decreases, the tree grows and
in order to keep the right number of leaves, the intensity for choosing
them has to decrease; this can be done up to the lower bound
$\theta_\lambda$. 

By   considering   $L_\theta(\lambda)$   the   number   of   leaves   of
$\tau_\theta(\lambda)$,  we get   that
$\psi'(\theta)   L_\theta(\lambda)/    \psi_q(\psi^{-1}(\lambda))$   is   a   backward
martingale, see  Proposition \ref{prop:mart-l}. By taking the limit as $\lambda$ goes to infinity, and since
the    total    mass    of   $\bm^{\tau_\theta(\lambda)}$,    that    is
$L_\theta(\lambda)/  \psi_q(\psi^{-1}(\lambda))$,  converges  to  the  total  mass  of
$\bm^{\ct_\theta}$, say $\sigma_\theta$, we get in Proposition
\ref{prop:mart-s} that $\psi'(\theta)\sigma_\theta$ is also a backward
martingale. 

Then we consider the process $(\tau_\theta(\lambda),
\theta>\theta_\lambda)$ backward in time and consider its ascension time
 $A_\lambda$ defined in \reff{eq:def-Al} as the first time at which the
 tree $\tau_\theta(\lambda)$ is unbounded. Of course, this corresponds to
 the ascension time of $(\ct_\theta, \theta\in \Theta^\psi)$ when it is
 larger than $\theta_\lambda$. We give in Proposition \ref{prop:S-SL}
 the distribution of  $(\tau_\theta(\lambda), \theta\geq
 A_\lambda)$ and identify it in Proposition \ref{disequ} using the
 pruning of a tree $\ct^*_0(\lambda)$ with an infinite spine defined in Sections
 \ref{sec:infinite-CRT} and \ref{sec:dist-infinite-CRT}. We also prove
 the convergence, as $\lambda$ goes to infinity of the  tree $\ct^*_0(\lambda)$ 
toward the CRT $\ct^*_0$ with infinite spine introduced in
\cite{ad:ctvmp}. The latter can be seen as a sub-tree of L\'{e}vy trees
with immigration, see \cite{cw:ggwt} for further work in this direction.

\section{L\'evy trees and the forest obtained by pruning}
\label{sec:levy}
\subsection{Notations}

Let $(E,d)$ be a metric  Polish space.  We denote by $\cm_{f}(E)$ (resp.
$\cm_f^{\text{loc}}(E)$) the space of all finite (resp.  locally finite)
Borel measures  on $E$.  For $x\in E$,  let $\delta_x$ denote  the Dirac
measure  at point  $x$.  For $\mu\in \cm_f^{\text{loc}}(E)$  and $f$  a
non-negative measurable function, we set $\langle \mu,f \rangle =\int
f(x) \, \mu(dx)= \mu(f)$. 

\subsection{Real trees}

We refer to \cite{e:prt} or  \cite{lg:rta} for a general presentation of
random  real trees.  Informally,  real trees  are metric  spaces without
loops, locally  isometric to  the real line.   More precisely,  a metric
space $(T,d)$ is a real tree if the following properties are satisfied:
\begin{enumerate}
	\item For every $s,t\in T$, there is a unique isometric map $f_{s,t}$
from $[0,d(s,t)]$ to $T$ such that $f_{s,t}(0)=s$ and $f_{s,t}(d(s,t))=t$.
	\item For every $s,t\in T$, if $q$ is a continuous injective map from
$[0,1]$ to $T$ such that $q(0)=s$ and $q(1)=t$, then
$q([0,1])=f_{s,t}([0,d(s,t)])$.
\end{enumerate}
If $s,t\in T$, we will note $\llbracket s,t \rrbracket$
the range of the
isometric map  $f_{s,t}$ described above and  $\llbracket s,t
\llbracket$  for $\llbracket s,t \rrbracket \backslash \{t\}$. 

We say that $(T,d,\emptyset)$ is a rooted real tree with root
$\emptyset$ if $(T,d)$ is a real tree and $\emptyset\in T$ is a
distinguished vertex.

Let $(T,d,\emptyset)$ be a rooted real tree.  The degree $n(x)$ of $x\in
T$ is  the number of  connected components of $T\setminus\{x\}$  and the
number of children of $x\neq  \emptyset$ is $\kappa_x=n(x)-1$ and of the
root is  $\kappa_\emptyset=n(\emptyset)$.  We shall consider  the set of
leaves $ \Lf (T)=\{x\in T,\,  \kappa_x=0\}$, the set of branching points
$\Br(T)=\{x\in T, \, \kappa_x\geq 2\}$ and the set of infinite branching
points $\Br_\infty(T) =  \{ x\in T,\, \kappa_x = \infty \}  $.  We say
that a
tree is  discrete if $\{x\in \Lf(T)\cup  \Br(T); d(\emptyset,x)\leq a\}$
is finite for all  $a$. The skeleton of $T$ is the  set of points in the
tree that aren't leaves: $\Sk(T)=T\backslash \Lf (T)$.  The trace of the
Borel $\sigma$-field of  $T$ restricted to $\Sk(T)$ is  generated by the
sets  $\llbracket  s,s'  \rrbracket$;  $s,s' \in  \Sk(T)$.   One
defines  uniquely a  $\sigma$-finite  Borel measure  $\ell^{T}$ on  $T$,
called the length measure of $T$, such that:
\[
\ell^{T}(\Lf (T)) = 0
\quad\text{and}\quad
 \ell^{T}(\llbracket s,s'
\rrbracket)=d(s,s').
\]

For every $x\in T$, $\lb\emptyset ,x\rb$ is
interpreted as the ancestral line of vertex $x$ in the tree. We define
a partial order on $T$ by setting $x\preccurlyeq y$ ($x$ is an
ancestor of $y$) if $x\in\lb \emptyset,y\rb$.
If $x,y\in T$, there exists a unique $z\in T$, called the Most Recent
Common Ancestor (MRCA) of $x$ and $y$,  such that $\lb
\emptyset,x\rb\cap\lb\emptyset,y\rb=\lb\emptyset,z\rb$, and we write
$z=x\wedge y$.

\subsection{Measured rooted real trees}
According to \cite{adh:ghptwms}, one can define a
Gromov-Hausdorff-Prohorov metric on the space of rooted measured metric
space as follows.

\par Let $(X,d)$ be a Polish metric space. For $A,B\in \cb(X)$, we set:
\[
 d_\text{H}(A,B)= \inf \{ \varepsilon >0,\ A\subset B^\varepsilon\
 \mathrm{and}\ B\subset
A^\varepsilon \},
\]
the Hausdorff distance between $A$ and $B$, where $A^\varepsilon = \{ x\in X,
\inf_{y\in A} d(x,y) < \varepsilon\}$ is the $\varepsilon$-halo set of $A$.
If $\mu,\nu \in \cm_f(X)$, we set:
\[
 d_\text{P}(\mu,\nu) = \inf \{ \varepsilon >0,\ \mu(A)\le \nu(A^\varepsilon) +
 \varepsilon
\text{ and }
\nu(A)\le \mu(A^\varepsilon)+\varepsilon\
\text{ for all closed set } A \},
\]
the Prohorov distance between $\mu$ and $\nu$.

A rooted measured  metric space $\cx = (X,d,  \emptyset,\mu)$ is a
metric space $(X , d)$ with a distinguished element $\emptyset\in X$
and  a locally  finite Borel  measure  $\mu\in \cm_f^{\text{loc}}(E)$.
Two  rooted measured  metric spaces  $\cx=(X,d,\emptyset,\mu)$ and
$\cx'=(X',d',\emptyset',\mu')$  are  said  GHP-isometric  if
there exists  an isometric bijection  $\Phi:X \rightarrow X'$  such that
$\Phi(\emptyset)= \emptyset'$ and $\Phi_* \mu = \mu'$, where
$\Phi_* \mu$ is the measure $\mu$ transported by $\Phi$.

Let $\cx=(X,d,\emptyset,\mu)$ and $\cx'=(X',d',\emptyset',\mu')$ be two compact
rooted measured metric spaces, and define:
\[
 d_{\text{GHP}}^c(\cx,\cx') = \inf_{\Phi,\Phi',Z} \left(
 d_\text{H}^Z(\Phi(X),\Phi'(X')) +
d^Z(\Phi(\emptyset),\Phi'(\emptyset')) + d_\text{P}^Z(\Phi_* \mu,\Phi_*'
\mu') \right),
\]
where the infimum is taken over all isometric embeddings $\Phi:X\hookrightarrow
Z$ and $\Phi':X'\hookrightarrow Z$ into some common Polish metric space
$(Z,d^Z)$.

If $\cx=(X,d,\emptyset,\mu)$ is a rooted measured metric space, then for
$r\geq 0$  we will consider  its restriction to  the ball of  radius $r$
centered   at  $\emptyset$,  $\cx^{(r)}=(X^{(r)},   d^{(r)},  \emptyset,
\mu^{(r)})$, where
\[
X^{(r)}=\{x\in X; d(\emptyset,x)\leq r\},
\]
the metric $d^{(r)}$ is the restriction of $d$ to $X^{(r)}$, and the measure
$\mu^{(r)}(dx)=\ind_{X^{(r)}} (x)\; \mu(dx)$ is the restriction of $\mu$ to
$X^{(r)}$.

We will  denote by  $\T$ the set  of (GHP-isometry classes  of) measured
rooted real trees $(T, d, \emptyset, \bm)$ where $(T,d, \emptyset)$ is a
locally compact  rooted real tree and $  \bm\in\cm_f^\text{loc}(T)$ is a
locally  finite measure  on $T$.   Sometimes,  we will  write $(T,  d^T,
\emptyset^T,  \bm^T)$ for  $(T,  d, \emptyset,  \bm)$  to stress  the
dependence in $T$. Sometimes, when there is no confusion, we will simply
write  $T$ for  $(T, d,  \emptyset,  \bm)$ and  $\tilde T$  for $(T,  d,
\emptyset)$.   We define  the following  function on  $\T^2$,  for $T_1,
T_2\in \T$:
\[
 d_{\text{GHP}}(T_1,T_2) = \int_0^\infty \expp{-r} \left(1 \wedge
d^c_{\text{GHP}}\left(T_1^{(r)},T_2^{(r)}\right)
\right) \ dr.
\]
According   to  Corollary  2.8   in  \cite{adh:etiltvp},   the  function
$d_{\text{GHP}}$ is well defined  and $(\T, d_{\text{GHP}})$ is a Polish
metric space.

\subsection{Grafting procedure}
\label{sec:graft}
We  will define  in this  section a  procedure by  which we  add (graft)
measured  rooted  real  trees   on  an  existing  measured  rooted  real
trees. More  precisely, let $T$ be  a measured rooted real  tree and let
$((T_i,x_i),i\in  I)$ be  a finite  or countable  family of  elements of
$\T\times T$.   We define the real  tree obtained by  grafting the trees
$T_i$  on $T$  at  point  $x_i$. We  set  $\hat {T}  =  T \sqcup  \left(
  \bigsqcup_{i\in  I} T_i\backslash\{\emptyset^{T_i}\}  \right)  $ where
the symbol $\sqcup$  means that we choose for  the sets $(T_i)_{i\in I}$
representatives  of  GHP-isometry classes  in  $\T$  which are  disjoint
subsets of some common set and that we perform the disjoint union of all
these sets. We set $\emptyset^{\hat T}=\emptyset^T$. The set $\hat T$ is
endowed with the following metric $d^{\hat T}$: if $s,t\in \hat T$,
\begin{equation*}
d^{\hat  T} (s,t) =
\begin{cases}
d^T(s,t)\ & \text{if}\ s,t\in T, \\
d^T(s,x_i)+d^{T_i}(\emptyset^{T_i},t)\ & \text{if}\
s\in T,\ t\in T_i\backslash\{\emptyset^{T_i}\} , \\
d^{T_i}(s,t)\ & \text{if}\ s,t\in T_i\backslash\{\emptyset^{T_i}\} ,\\
d^T(x_i,x_j)+d^{T_j}(\emptyset^{T_j},s)+d^{T_i}
(\emptyset^{T_i},t)\
& \text{if}\ i\neq j \ \text{and}\ s\in T_j\backslash\{\emptyset^{T_j}\} ,\ t\in
T_i\backslash\{\emptyset^{T_i}\} .
\end{cases}
\end{equation*}
We define the mass measure on $\hat  T$ by:
\[
\mathbf{m}^{\hat  T}=\mathbf{m}^T+\sum_{i\in I}\left(\ind_{
 T_i\backslash\{\emptyset^{T_i}\}} \mathbf{m}^{T_i}+
\mathbf{m}^{T_i}(\{\emptyset^{T_i}\}) \delta_{x_i}\right).
\]
It is clear that the rooted metric
space $(\hat {T},d^{\hat  T},\emptyset^{\hat  T})$ is still a rooted complete
real tree. (Notice that it is not always true that $\hat {T}$ remains locally
compact  or that
$\mathbf{m}^{\hat  T}$ defines a
locally finite measure on $\hat  T$).
 We will use the following notation for the grafted tree:
\begin{equation}
T \circledast_{i\in I}(T_i,x_i) = (\hat  T,d^{\hat  T},\emptyset^{\hat  T},
\mathbf{m}^{\hat  T} ) ,
\label{def:gref}
\end{equation}
where we make the convention that $T \circledast_{i\in I}(T_i,x_i)=T$
for $I=\emptyset$. If $\varphi$ is an isometry from $T$ onto $T'$, then
$T \circledast_{i\in I}(T_i,x_i) $ and $T' \circledast_{i\in
  I}(T_i,\varphi(x_i)) $ are also isometric. Therefore, the grafting
procedure is well defined on $\T$.

In Section \ref{sec:GW}, we shall use the grafting procedure for rooted
real trees but without mass measure. Recall $\tilde T=(T, \emptyset^T,
d^T)$. We shall use the following
notation:
\begin{equation}
\tilde T \tilde  \circledast_{i\in
  I}(\tilde T_i,x_i) =(\hat  T,d^{\hat
  T},\emptyset^{\hat  T} ) , 
\label{def:gref-m}
\end{equation}
where we also make the convention that $\tilde T \tilde  \circledast_{i\in
  I}(\tilde T_i,x_i) =\tilde T  $ for $I=\emptyset$.

\subsection{Sub-trees above a given level}
 For $T\in \T$ we set $H_{\text{max}}(T)=\sup_{x\in T}
 d^T(\emptyset^T,x)$ the height of $T$ and for
$a\geq 0$:
\[
T^{(a)}=\{x\in T, \, d(\emptyset,x)\le a\}
\quad\mbox{and}\quad T(a)=\{x\in T, \, d(\emptyset,x)=a\}
\]
the restriction of the tree $T$ under level $a$ and the set of
vertices of $T$ at level $a$ respectively.   We denote by  $(T^{i,\circ},i\in I)$ the connected components of
$T\setminus  T^{(a)}$.   Let  $\emptyset_i$  be  the MRCA  of  all  the
vertices    of    $T^{i,\circ}$.    We    consider    the   real    tree
$T^i=T^{i,\circ}\cup\{\emptyset _i\}$ rooted at point $\emptyset_i$ with
mass  measure  $\bm^{T^i}$ defined  as  the  restriction  of $\bm^T$  to
$T^{i,\circ}$. Notice that $T=T^{(a)} \circledast_{i\in I}(T_i,\emptyset_i) $.
We will consider the point measure on $T\times \T$:
\begin{equation}
   \label{eq:def-cna}
\cn_a^T=\sum_{i\in I}\delta_{(\emptyset_i,T^i)}.
\end{equation}

\subsection{Excursion measure of a L\'evy tree}
\label{sec:exc-meas}
Let $\alpha\in \R$, $\beta\ge 0$ and $\pi$ be  a $\sigma$-finite measure
on      $(0,+\infty)$     such      that     $\int_{(0,+\infty)}(r\wedge
r^2)\pi(dr)<+\infty$.
The   branching mechanism $\psi$ with characteristic $(\alpha,\beta,
\pi)$ is
defined by:
\begin{equation}
   \label{eq:psi}
\psi(\lambda)=\alpha\lambda+\beta\lambda^2
+\int_{(0,+\infty)}\left(\expp{-\lambda
  r}-1+\lambda r\right)\pi(dr).
\end{equation}

We assume the following assumptions:
\begin{enumerate}
\item[(H1)]  The   branching  mechanism   $\psi$  is  conservative:   for  all
  $\varepsilon>0$,
\[
\int_0^\varepsilon \frac{d\lambda}{|\psi(\lambda)|}=+\infty.
\]
The  conservative assumption  is  equivalent to  the  finiteness of  the
corresponding CSBP at all time.
\item[(H2)] The Grey condition holds:
\begin{equation}
\int^{+\infty}\frac{d\lambda}{\psi(\lambda)}<+\infty.
\end{equation}
The  Grey  condition  is  equivalent  to the  a.s.   finiteness  of  the
extinction time  of the corresponding  CSBP. This assumption is  used to
ensure that the  corresponding L\'{e}vy tree is compact.
\item[(H3)] $\beta>0$ or $\int_{(0,1)} \ell \pi(d\ell)=+\infty $. This
  condition is equivalent to the fact that the L\'{e}vy process with index
  $\psi$ is of infinite variation (and the L\'{e}vy tree is not discrete).
 \end{enumerate}

Let $v$ be
the unique non-negative solution of the equation:
\[
\int_{v(a)}^{+\infty}\frac{d\lambda}{\psi(\lambda)}=a.
\]

Results  from  \cite{dlg:pfalt} in  the  (sub)critical  case, using  the
coding of compact real trees by  height function, can be extended to the
super-critical case,  see \cite{adh:etiltvp}. They can be  stated in the
following form.  There  exists a $\sigma$-finite measure $\N^\psi[d\ct]$
on $\T$,  or excursion  measure of a  L\'{e}vy tree, with  the following
properties.
\begin{enumerate}
   \item[(i)] \textbf{Height.} For all $a>0$,
     $\N^\psi[H_{\text{max}}(\ct)>a]=v(a)$.

\item[(ii)] \textbf{Mass measure.} The mass measure $\bm^\ct$ is supported by
$\Lf (\ct)$, $\N^\psi[d\ct]$-a.e.

   \item[(iii)] \textbf{Local time.}
There exists a $\ct$-measure valued process $(\ell^a, a\geq 0)$
c\`adl\`ag for the weak topology on finite measure on $\ct$ such that
$\N^\psi[d\ct]$-a.e.:
\begin{equation}
\label{eq:int-la}
\mathbf{m}^{\ct}(dx) = \int_0^\infty \ell^a(dx) \, da,
\end{equation}
$\ell^0=0$, $\inf\{a > 0 ; \ell^a = 0\}=\sup\{a \geq 0 ; \ell^a\neq
0\}=H_{\text{max}}(\ct)$ and for every fixed $a\ge 0$,
$\N^\psi[d\ct]$-a.e.:
 \begin{itemize}
 \item The measure $\ell^a$ is supported on
$\ct(a)$.
 \item We have for every bounded
continuous function $\phi$ on $\ct$:
\begin{align*}
\langle\ell^a,\phi \rangle
& = \lim_{\varepsilon \downarrow 0}
\frac{1}{v(\varepsilon)} \int \phi(x) \ind_{\{H_{\text{max}}(\ct')\ge
\varepsilon\}} \cn_a^{\ct}(dx, d\ct') \\
 & = \lim_{\varepsilon \downarrow 0} \frac{1}{v(\varepsilon)} \int \phi(x)
\ind_{\{H_{\text{max}}(\ct')\ge \varepsilon\}}
\cn_{a-\varepsilon}^{\ct}(dx, d\ct'),\ \text{if}\ a>0.
\end{align*}
 \end{itemize}
 Under $\N^\psi$,  the real  valued process $(\langle\ell^a,1  \rangle ,
 a\geq  0)$ is  distributed as  a CSBP  with branching  mechanism $\psi$
 under its canonical measure.
\item[(iv)]  \textbf{Branching property.}
For every $a>0$, the conditional
distribution of the point measure $\cn_a^{\ct}(dx,d\ct')$ under
$\N^\psi[d\ct|H_{\text{max}}(\ct)>a]$, given $\ct^{(a)}$, is that of a Poisson
point measure on $\ct(a)\times \T$ with intensity
$\ell^a(dx)\N^\psi[d\ct']$.
\item[(v)] \textbf{Branching points.}
\begin{itemize}
\item $\N^\psi[d\ct]$-a.e., the branching points of $\ct$ have 2
  children or an
  infinity number of children.
\item The set of binary branching points (i.e. with 2 children) is empty
  $\N^\psi$ a.e if $\beta=0$ and is a countable dense subset of $\ct$
  if $\beta>0$.
\item The set $\Br_\infty(\ct)$  of infinite branching points is
  nonempty with $\N^\psi$-positive measure if and only if $\pi\ne 0$. If
  $\langle  \pi,1\rangle=+\infty$, the set  $\Br_\infty(\ct)$ is
  $\N^\psi$-a.e. a countable dense subset of $\ct$.
\end{itemize}
\item[(vi)] \textbf{Mass of the nodes.}
The set $\{ d(\emptyset,x),\ x\in \Br_\infty(\ct) \}$ coincides
$\N^\psi$-a.e. with the set of discontinuity times of the mapping $a\mapsto
\ell^a$. Moreover, $\N^\psi$-a.e., for every such discontinuity time $b$, there
is a unique $x_b\in \Br_\infty(\ct)\cap\ct(b)$ and
$\Delta_b>0$, such that:
\[
\ell^b = \ell^{b-} + \Delta_b \delta_{x_b},
\]
where $\Delta_b>0$ is called the mass of the node $x_b$. Furthermore
$\Delta_b$ can  be obtained
by the approximation:
\begin{equation}
\Delta_b = \lim_{\varepsilon \rightarrow 0}
\frac{1}{v(\varepsilon)}
n(x_b,\varepsilon),
\label{DefMas}
\end{equation}
where $n(x_b,\varepsilon)=\int \ind_{\{x_b\}}(x)\ind_{\{H_{\text{max}}(\ct') >
\varepsilon\}} \cn_b^\ct(dx,d\ct')$ is the number of sub-trees with  MRCA
$x_b$ and height larger than $\varepsilon$.
\end{enumerate}

In order to stress the dependence in $\ct$, we may write $\ell^{a, \ct}$
for $\ell^a$.
We set $\sigma^\ct$  or simply $\sigma$ when there  is no confusion, for
the
total mass of the mass measure on $\ct$:
\begin{equation}
\label{eq:s=la}
\sigma=\bm^{\ct}(\ct).
\end{equation}
Notice that \reff{eq:int-la} readily implies that $\bm^\ct(\{x\})=0$ for
all $x\in \ct$.

\subsection{Related measure on L\'evy trees}
\label{sec:meas-LT}

We  define a  probability  measure on  $\T$  as follow.   Let $r>0$  and
$\sum_{k\in\ck}\delta_{\ct^k}$ be  a Poisson point measure  on $\T$ with
intensity  $r\N^\psi$.   Consider $\emptyset$  as  the trivial  measured
rooted  real tree reduced  to the  root with  null mass  measure. Define
$\ct=\emptyset   \circledast_{k\in  \ck}(  \ct^k,   \emptyset)$.   Using
Property (i) as well as \reff{eq:N1-es} below, one easily get that $\ct$
is  a measured locally  compact rooted  real tree,  and thus  belongs to
$\T$. We denote by $\P^\psi_r$ its distribution. Its corresponding local
time  and mass  measure are  respectively defined  by $\ell^a=\sum_{k\in
  \ck}  \ell^{a, \ct^k}$  for  $a\geq 0$,  and $\bm^\ct=\sum_{k\in  \ck}
\bm^{\ct^k}$.    Furthermore,   its   total    mass   is    defined   by
$\sigma=\sum_{k\in  \ck}  \sigma^{\ct^k}$.   By  construction,  we  have
$\P^\psi_r(d\ct)$-a.s.          $\emptyset\in          \Br_\infty(\ct)$,
$\Delta_\emptyset=r$  (see  definition  \reff{DefMas}  with  $b=0$)  and
$\ell^0=r\delta_\emptyset$.   Under $\P^\psi_r$  or under  $\N^\psi$, we
define the process $\cz=(\cz_a,a\geq 0)$ by:
\[
\cz_a=\langle\ell^a,1 \rangle .
\]
According to Property (iii), under $\P^\psi_r$ (resp. $\N^\psi$), the real
valued process $\cz$ is distributed as a
CSBP with branching mechanism $\psi$ with initial value $r$ (resp. under
its canonical measure). Notice
that (under $\N$ or $\P_r^\psi$):
\begin{equation}
   \label{eq:sigma=intZ}
\sigma=\int_0^{+\infty } \cz_a \; da =\bm^\ct(\ct).
\end{equation}
In particular, as $\sigma$ is distributed as the total mass of a CSBP
under its canonical measure, we have that $\N^\psi$-a.s. $\sigma>0$ and
for $q> 0$ such that $\psi(q)>0$:
\begin{equation}
   \label{eq:N1-es}
\N^\psi\left[1-\expp{-\psi(q) \sigma} \right]=q
\quad \text{and}\quad
 \N^\psi\left[\sigma\expp{-\psi(q)\sigma}\right]=\frac{1}{\psi'(q)} .
\end{equation}
The last equation holds for $q=0$ if $\psi'(0)>0$.

We will consider the following measures on $\T$:
\begin{equation}
   \label{eq:cn}
\cn_\theta^\psi
[d\ct]=2\beta\theta\N^{\psi}[d\ct]
 +\int_{(0,+\infty )}  {\pi}(dr)(1-\expp{-\theta r})\P_r^{\psi}(d\ct)
\end{equation}
and
\begin{equation}
\label{eq:def-bN}
\bN^\psi[d\ct]=\frac{\partial}{\partial\theta} \; \cn_\theta^\psi
[d\ct]_{|_{\theta=0}}=2\beta \N^\psi[d\ct]+\int_0^{+\infty}r\pi(dr)\,\P_r^\psi(d\ct).
\end{equation}

Elementary computations yield for $q>0$ such that $\psi(q)>0$:
\begin{equation}
   \label{eq:bN-moment}
\cn_\theta^\psi\left[1- \expp{-\psi(q)
    \sigma}\right]
=\psi(\theta+q)- \psi(\theta) -\psi(q)
\quad\text{and}\quad
\bN^\psi\left[1-\expp{-\psi(q) \sigma}\right]=\psi'(q)
- \psi'(0).
\end{equation}

\subsection{Girsanov transformation}
\label{sec:Girsanov}
For   $\theta\in   \R$,    we   set   $\pi_\theta(dr)=\expp{-\theta   r}
\pi(dr)$.  Let  $\Theta'$ be  the  set of  $\theta\in  \R$  such that  $
\int_{(1,+\infty)}  \pi_\theta   (dr)  <  +\infty$.   If  $\pi=0$,  then
$\Theta'=\R$. We also set $\theta_\infty  = \inf \Theta'$. It is obvious
that $[0,+\infty  )\subset \Theta'$,  $\theta_\infty \leq 0$  and either
$\Theta'=[\theta_\infty  ,  +\infty   )$  or  $\Theta'=(\theta_\infty  ,
+\infty   )$.    We  introduce   the   following  branching   mechanisms
$\psi_\theta$ for $\theta\in \Theta'$:
\begin{equation}
 \label{eq:def-psi-q}
\psi_\theta(\lambda)=\psi(\lambda+\theta)-\psi(\theta), \quad
\lambda+\theta\in \Theta',
\end{equation}
with characteristic:
\begin{equation}
   \label{eq:characteristic}
\left(\psi'(\theta), \beta, \pi_\theta\right).
\end{equation}

Let $\Theta^\psi$ be the set of $\theta\in \Theta'$ such that $\psi_\theta$ is
conservative. Obviously, we have:
\[
[0,+\infty )\subset \Theta^\psi \subset \Theta' \subset \left(\Theta^\psi
\cup \{\theta_\infty\}\right).
\]
Let  $\theta^*$   be  the  unique   positive  root  of  $\psi'$   if  it
exists.  Notice  that  $\theta^*=0$  if  $\psi$  is  critical  and  that
$\theta^*$  exists  and is  positive  if  $\psi$  is super-critical.  If
$\theta^*$  exists, then  the branching  mechanism  $\psi_{\theta^*}$ is
critical. We set $\Theta_*^\psi$ for $[\theta^*,+\infty )$ if $\theta^*$
exists and $\Theta_*^\psi=\Theta^\psi$ otherwise. The function $\psi$ is
a one-to-one mapping from $\Theta_*^\psi$ onto $\psi(\Theta^\psi_*)$. We
write  $\psi^{-1}$  for  the   inverse  of  the  previous  mapping.
In particular, if $\psi_\theta$ is (sub)critical then  we have
$\psi^{-1}(\psi(\theta))=\theta$;  and if $\psi_\theta$ is super-critical then we have
$\theta<\theta^*< \psi^{-1}(\psi(\theta))$.
We set:
\begin{equation}
   \label{eq:q_0}
q_0=\psi^{-1}(0).
\end{equation}

Note that if $\psi$ is super-critical, then $q_0>0$ and, thanks to
\reff{eq:N1-es},
$\N^{\psi}[\sigma=+\infty]=\psi^{-1}(0)>0$.

We recall  the Girsanov transformation from  \cite{ad:ctvmp}, which sums
up the  situation for any branching  mechanism $\psi$. Let  $\psi$ be a
branching mechanism  satisfying (H1-3). Let  $\theta\in \Theta^\psi$ and
$a>0$. We set:
\[
M_a^{\psi, \theta}=\exp\left\{\theta \cz_0-
  \theta\cz_a-\psi(\theta)\int_0^a\cz_sds\right\}.
\]
Recall that $\cz_0=0$ under $\N^\psi$.
For any non-negative measurable functional $F$ defined on $\T$, we have
for $\theta\in \Theta^\psi$ and $a\geq 0$:
 \begin{equation}
 \label{eq:GPx}
 \E_r^{\psi_\theta}[F(\ct^{(a)})]
= \E_r^{\psi} \left[F(\ct^{(a)})
M_a^{\psi,\theta} \right]
\quad\text{and}\quad
 \N^{\psi_\theta}[F(\ct^{(a)})]
= \N^{\psi} \left[F(\ct^{(a)})
M_a^{\psi,\theta}  \right].
 \end{equation}
Furthermore, if $\theta\geq \theta^*$, then we have:
 \begin{align}
 \label{eq:GPr}
 \E_r^{\psi_\theta}[F(\ct)]
&= \E_r^{\psi} \left[F(\ct)\expp{\theta r - \psi(\theta) \sigma}
\ind_{\{\sigma<+\infty \}}  \right],\\
\label{eq:GN}
 \N^{\psi_\theta}[F(\ct)]
&= \N^{\psi} \left[F(\ct)\expp{- \psi(\theta) \sigma}
\ind_{\{\sigma<+\infty \}}\right],\\
\label{eq:Girsanov-bN}
\bN^{\psi_\theta}[F(\ct)]&=\bN^\psi
\left[F(\ct)\expp{-\psi(\theta)\sigma}\ind_{\{\sigma<+\infty \}}\right].
 \end{align}

 We   have   that   under   $\P_r^{\psi}(d\ct)$,  the   random   measure
 ${\cn}_0^{\ct}(dx, d\ct')$, defined by \reff{eq:def-cna} with $a=0$, is
 a  Poisson  point  measure  on $\{\emptyset\}\times\T$  with  intensity
 $r\delta_{\emptyset}(dx)\mN^{\psi}[d\ct']$.   Then,   using  the  first
 equality   in    \reff{eq:GPx}   with   $F=1$,   we    get   that   for
 $\theta\geq\theta^*$ and $a>0$,
\begin{equation}
   \label{eq:Nmq=-q}
 \N^{\psi_{\theta}}\left[1-\exp\left\{\theta\Z_{a}+\psi(\theta)
     \int_0^a\Z_sds\right\}\right]=-\theta.
\end{equation}

\subsection{Pruning L\'evy trees and CRT-valued Processes}
\label{sec:pruning}

A   general   pruning  of   a   L\'evy   tree   has  been   defined   in
\cite{adv:plcrt}. Under  $\N^\psi[d\ct]$ and conditionally  on $\ct$, we
consider  a mark  process $M^\ct(d\theta,dy)$  on  the tree  which is  a
Poisson point measure on $\R_+\times \ct$ with intensity:
\[
\ind_{[0,+\infty)}(\theta)d\theta\left(2\beta \ell^\ct(dy)+\sum_{x\in
   \Br_\infty(\ct)}\Delta_x\delta_x(dy)\right).
\]
The atoms $(\theta_i,y_i)_{i\in I}$ of this measure can be seen as
marks that arrive on the tree, $y_i$ being the location of the mark
and $\theta_i$ the ``time'' at which it appears. There are two kinds
of marks: some are ``uniformly'' distributed on the skeleton of the
tree (they correspond to the term $2\beta \ell^\ct$ in the intensity)
whereas the others lay on the infinite branching points of the tree: an infinite
branching point $y$ being first marked after an exponential time with
parameter $\Delta_y$.

For every $x\in\ct$, we set:
\[
\theta(x)=\inf\{\theta>0,\ M^\ct([0,\theta]\times \lb\emptyset,x\rb)>0\},
\]
which  is  called  the  record   process  on  the  tree  as  defined  in
\cite{ad:rpcrt}.  This corresponds  to the  first  time at  which a  mark
arrives on  $\lb\emptyset,x\rb$.  Using  this record process,  we define
the pruned tree at time $q$ as:
\[
\ct_q=\{x\in\ct,\ \theta(x)\ge q\}
\]
with  the  induced  metric,   root  $\emptyset$  and  mass  measure  the
restriction of the mass measure $\bm^\ct$.  If one cuts the tree $\ct$ at
time $\theta_i$  at point  $y_i$, then $\ct_q$  corresponds to the
resulting sub-tree  of $\ct$
containing the root at time $q$. According to \cite{adv:plcrt}, Theorem
1.1, for fixed $q>0$, the distribution of $\ct_q$ under $\N^\psi$
is $\N^{\psi_q}$. We set:
\[
\sigma_{q}=\bm^{\ct_q}(\ct_q).
\]

Because of  the pruning procedure,  we have $\ct_{\theta}\subset\ct_{q}$
for $0\leq q\leq\theta$.  The  tree-valued process $(\ct_q,q\ge 0)$ is a
Markov  process  under  $\N^\psi$,  see  \cite{ad:ctvmp}.   The  process
$(\ct_q,q\ge  0)$ is  a  non-increasing process  (for  the inclusion  of
trees), and is c\`adl\`ag.  We recall the transition probabilities for
the time reversed process which
are given by the so-called special Markov property (see \cite{adv:plcrt}
Theorem 4.2  or \cite{ad:ctvmp}  Theorem 5.6).
\begin{theo}
   \label{theo:markov-special}
Let $\psi$ be a branching mechanism satisfying (H1-3). Let  $0\leq q\leq\theta$
and   $\ct_\theta$    distributed   according   to   $\N^{\psi_\theta}$.
Conditionally on  $\ct_\theta$, let $\sum_{i\in I^{\theta,q}}\delta_{(x_i,
  \ct^{i}_q)}$ be  a Poisson point  measure on $\ct_{\theta}\times\mbb
T$ with intensity:
\[
\bm^{\ct_\theta}(dx) \cn^{\psi_q}_{\theta-q}[d\ct].
\]
Then, under $\N^\psi$, $(\ct_\theta, \ct_q)$ is distributed as:
\[
\left(\ct_\theta,\, \ct_{\theta}\circledast_{i\in I^{\theta,q}}(
  \ct^{i}_q, x_i)
\right).
\]
\end{theo}
According to  \reff{eq:characteristic}, the intensity $\cn^{\psi_q}_{\theta-q}$
is  given  by  \reff{eq:cn}  with  $\psi$  replaced  by  $\psi_q$  and
$\pi(dr)$ replaced by $\expp{-qr}\pi(dr)$, that is:
\begin{equation}
   \label{eq:cn-q}
\cn_{\theta-q}^{\psi_q}
[d\ct]=2\beta(\theta-q)\N^{\psi_q}[d\ct]
 +\int_{(0,+\infty )}  \expp{-qr} {\pi}(dr)(1-\expp{-(\theta-q)
   r})\P_r^{\psi_r}(d\ct).
\end{equation}
The time-reversed process is a Markov process and its
infinitesimal transitions are described in \cite{adh:etiltvp}.

\section{Sub-tree Processes}
\subsection{Sub-tree of the L\'{e}vy tree}
\label{sec:subtree}

Following \cite{dw:glt}, we define a sub-tree process  obtained from pruned CRTs
and Poissonian  selection of leaves.  Let $\psi$ be a branching
mechanism satisfying (H1-3). We set:
\begin{equation}
   \label{eq:def-eta}
\eta=\psi^{-1}(\lambda)\quad\text{for}\quad \lambda\geq 0.
\end{equation}
Notice that $\psi(\lambda)=\eta$ and, with $q_0$ defined by
\reff{eq:q_0},  $\eta>q_0$  if $\lambda>0$.

Conditionally on the tree-valued process $(\ct_\theta,
\theta\in  \Theta^\psi)$,  we consider the point measure of marked
leaves of $\ct_0$:
\begin{equation}
   \label{eq:def-P0}
\cp^0  (dt,dx)=\sum_{i\in  I_0}\dz_{(t_i,
  x_i)}
\end{equation}
defined as a Poisson point measure on $\R_+\times \ct_0 $ with intensity
measure $ dt\, {\bf m}^{\ct_0  }(dx)$. We set:
\[
M_\lambda=\cp^0([0,\lambda]\times\ct_0)
\]
the number of marked leaves in $\ct_0$. We shall be working on
$\{M_\lambda\geq 1\}$ and consider the probability measure:
\begin{equation}
   \label{eq:def-P-psi-l}
\P^{\psi,\lambda}(d\ct)=\N^\psi[d\ct\,|\, M_\lambda\geq 1].
\end{equation}
Notice that $\eta=\N^\psi[M_\lambda\geq 1]$. 
We might  write $\cp^\theta
(dt,dx)=\sum_{i\in  I_\theta}  \delta_{(t_i,  x_i)}(dt,  dx) $  for  the
restriction  of $\cp^0$  to $\R_+\times  \ct_\theta$ for  $\theta\geq 0$.
On $\{M_\lambda\geq 1\}$, for  $\theta\geq0$ ,   we  define  the  pruned  sub-tree
$\tau_\theta(\lambda)$  containing the  root  and all  the ancestors  in
$\ct_\theta$ of the marked leaves of $\ct_0$:
\begin{equation}
   \label{subtree}
\tau_{0}(\lambda)=\bigcup_{i\in I_0, t_i\leq  \lambda} \llb\emptyset,  x_i\rrb
\quad\text{and}\quad 
\tau_{\theta}(\lambda)=\tau_{0}(\lambda)\bigcap \ct_\theta
\end{equation}
if $\lambda>0$, and if $\lambda=0$, we set:
\[
\tau_{\theta}(0)=\bigcap_{\lambda>0} \tau_{\theta}(\lambda).
\]
Notice  that  $\tau_\theta(0)=0$ if  $\ct_0$  has  finite mass measure,  whereas
$\tau_\theta(0)  \neq \emptyset$  (and $\tau_\theta(0)$  has  no  leaf)
if $\ct_0$  has infinite  mass.  By construction, we have a.s. that
$\tau_\theta(\lambda)$ is compact if and only if $\ct_\theta$ is compact
(that is $\ct_\theta$ has finite mass measure). The sub-tree  $\tau_{\theta}(\lz)$  of
$\ct_{\theta}$ and thus  of $\ct_0$ is endowed with  the obvious metric.
We shall consider the following mass measure on $\tau_\theta(\lambda)$:
\begin{equation}
   \label{eq:mass-tql}
\bm^{\tau_\theta(\lambda)}=\inv{\psi_\theta(\eta)} \sum_{x\in
  \Lf (\tau_\theta(\lambda))} \delta_x.
\end{equation}
As  $\theta$ varies, we  obtain a
sub-tree         process         with        parameter         $\lambda$:
$\tau(\lambda)=(\tau_{\theta}(\lz),\theta\geq   0    )$   which   is   a
non-decreasing tree-valued  stochastic process, that  is for $q<\theta$,
$\tau_{\theta}(\lz)\subset\tau_{q}(\lz)$.

\subsection{Reconstruction of the L\'{e}vy tree}
\label{sec:GW}


Let $g$ be  a generating function of a  distribution $p=(p(n), n\in \N)$
such that $g'(0)=0$ (\text{i.e.} $p(1)=0$) and let $c>0$. We shall define by
recursion a  Galton-Watson real tree with  reproduction distribution $p$
and branch length distributed according to an exponential random variable
with mean $1/c$.

Recall Notation \reff{def:gref-m} for the grafting procedure of trees
without mass measure.
We say that a discrete rooted real tree $\cg$ is
a $(g,c)$-Galton-Watson real tree if $\cg$ is distributed as:
\[
\lb
\emptyset,x\rb
\tilde \circledast_{1\leq k\leq K}(\cg_k,x) ,
\]
with:
\begin{itemize}
   \item $\lb
\emptyset,x\rb$  a real tree rooted at $\emptyset$ with no branching
point and such that $E_\emptyset=d(\emptyset,x)$ is
a random exponential variable with parameter $c$,
   \item  $K$ has generating function $g$ and is independent of
     $E_\emptyset$,
\item $(\cg_k, k\in \N^*)$ is a sequence of independent rooted real trees
which have the same distribution as $\cg$ and are independent of
$E_\emptyset$ and $K$.
 \end{itemize}

 Let   $\lambda\geq   0$   and   $\eta=\psi^{-1}(\lambda)$   such   that
 $\eta>0$. We consider the following generating function:
\begin{equation}
   \label{eq:def-g}
g_{(\psi,\lambda)}(r)=r+\frac{\psi((1-r)\psi^{-1}(\lambda))}
{\psi^{-1}(\lambda)\psi'(\psi^{-1}(\lambda))}= r
+\frac{\psi((1-r)\eta)}{\eta \psi'(\eta)}\cdot
\end{equation}
Notice that:
\begin{equation}
   \label{eq:g-01}
g'_{(\psi,\lambda)}(0)=0
\quad\text{and}\quad
g'_{(\psi,\lambda)}(1)=1 -\frac{\psi'(0)}{\psi'(\eta)}\cdot
 \end{equation} 

We  write $\cg(\psi,\lambda)$ for the $(g_{\psi,\lambda},
\psi'(\eta))$-Galton-Watson real tree. 
According to  Theorem 3.2.1
 in \cite{dlg:rtlpsbp}, if $\psi$ is (sub)critical, then the
discrete tree $\tau_0(\lambda)$  under $\P^{\psi,\lambda} $ is
distributed as a Galton-Watson tree $\cg(\psi,\lambda)$  with mass
measure given by \reff{eq:mass-tql}.
Furthermore,   we   can   reconstruct   the   L\'{e}vy   tree   $\ct$   from
$\tau_0(\lambda)$, thanks  to \cite{dw:glt}. For  this, recall Definition
\reff{eq:def-bN}  of $\bN$ and  define the  following probability measure on
$\R_+$:
\begin{equation}
   \label{eq:def-Gamma}
\Gamma_{d,\lambda}^{\psi}(dr)=\ind_{\{d=2\}}\frac{2\beta
}{\psi''(\eta )}\delta_0(dr)
+\frac{r^{d}\expp{-r\eta}}
{|\psi^{(d)}(\eta)|}\pi(dr)   .
\end{equation}

\begin{theo}[Theorem 5.6 of \cite{dw:glt}]
\label{theo:decomp-tree}
Assume that $\psi$ is (sub)critical and (H1-3) hold.  Let $\lambda>
0$ and $\eta=\psi^{-1}(\lambda)$. Under $\P^{\psi, \lambda}$ and conditionally on
$\tau_0(\lambda)$,
 $\ct_0$ is distributed as:
\[
\tilde \tau_0(\lambda)
\circledast_{i\in I}(\ct_i, x_i)
\circledast_{x\in \Br(\tau_0(\lambda))} (\ct'_x, x),
\]
with:
\begin{itemize}
   \item $\tilde \tau_0(\lambda)$ as $\tau_0(\lambda)$ but with 0 as mass
measure,
\item 
$\sum_{i\in I} \delta_{(x_i, \ct_i)}$ is a Poisson point measure
on $\ct_0(\lambda)\times \T$ with intensity
$\ell^{\tau_0(\lambda)}(dx)\, \bN^{\psi_\eta}[d\ct]$,
\item conditionally on 
$\sum_{i\in I} \delta_{(x_i, \ct_i)}$, the trees $\left(\ct'_x, x\in
  \Br(\tau_0(\lambda) )\right)$ are independent
  with $\ct'_x$ is distributed as 
\[
\int \Gamma^\psi_{\kappa(x),\lambda} (dr) \;
  \P_r^{\psi_\eta}[d\ct].
\]
 \end{itemize}
\end{theo}

\begin{remark}
  In fact, in Theorem 5.6 of \cite{dw:glt}, $\psi$ can be super-critical
  and $\lambda\geq  0$ with $\eta=\psi^{-1}(\lambda)>0$.  But  it is not
  obvious that in this  case the super-critical L\'evy tree distribution
  defined   in  \cite{dw:glt}   and  the   super-critical   L\'evy  tree
  distribution defined  in \cite{ad:ctvmp} and recalled  here in Section
  \ref{sec:exc-meas},  are in  fact the  same. However,  we  deduce from
  Remark \ref{rem:decomp-super} that this equality indeed holds.
\end{remark}

\section{Marginal distributions}
\label{SecsubMG}
The  main  goal  of  this   section  is  to  study  the  one-dimensional
distribution of  the sub-tree process $\tau(\lambda)=(\tau_{\theta}(\lz),
\theta\geq 0)$.

We first give an application of the special Markov property.

\begin{proposition}
\label{prop:marginal}
Let $\psi$ be a branching mechanism satisfying (H1-3).
Let $\lambda\geq 0$ and $\eta=\psi^{-1}(\lambda)$.
Under $\N^\psi$, the couple of trees $(\ct_\theta,
\tau_{\theta}(\lambda))$ on $\{M_\lambda\geq 1\}$ is distributed
as  $(\ct_0, \tau_0(\psi_{\theta}(\eta)))$ under $\N^{\psi_{\theta}}$ on $\{M_{\psi_{\theta}(\eta)}\geq 1\}$.
\end{proposition}

\begin{proof}
  We  first assume  $\lambda>0$.  From the  special  Markov property  of
  Theorem   \ref{theo:markov-special}   for   the  process   $({\mathcal
    T}(\theta), \theta\geq0)$ under $\mN^{\psi}$, we get:
\begin{equation}
   \label{eq:grafting-ADH}
\ct_0=\ct_{\theta}\circledast_{j\in J^{\theta,0}}(\ct^j_0, y_j),
\end{equation}
where    $\sum_{j\in   J^{\theta,0}}{\dz}_{(y_j,    \ct^j_0)}$   is,
conditionally   on   $\ct_\theta$,    a   Poisson   point   measure   on
$\ct_{\theta}\times\T$           with           intensity
$\bm^{\ct_{\theta}}(dy)\cn_\theta^\psi[d\ct].$

Recall  $\cp^\theta$  is  the  restriction to  $\R_+\times
\ct_\theta$  of  $\cp^0$  defined by \reff{eq:def-P0},  and thus  $\cp_1^\theta=\sum_{i\in  I_\theta} \delta_{x_i}
\ind_{\{t_i\leq \lambda\}} $ is a Poisson point measure on $\ct_\theta$ with
intensity $\lambda \bm^{\ct_\theta}(dx)$.

For $j\in J^{\theta,0}$, let  $s_j=\inf \{t_i; x_i\in \ct_0^j \text{ for
  $i\in I_0$}\}$.  Notice that  conditionally on $\ct_0^j$, $s_j$ has an
exponential       distribution       with       parameter       $\lambda
\bm^{\ct_0}(\ct_0^j)$.  We deduce  that, conditionally  on $\ct_\theta$,
$\cp_2^\theta=\sum_{j\in            J^{\theta,0}}           \delta_{y_j}
\ind_{\{s_j\leq \lambda\}}$    is    a     Poisson    point    measure    on
$\ct_{\theta}\times\R_+$ with intensity:
\[
\bm^{\ct_{\theta}}(dx)\cn_\theta^\psi\left[1- \expp{-\lambda
    \sigma}\right]
=\left[\psi(\theta+\eta)-\psi(\theta)-\psi(\eta)\right] \,
\bm^{\ct_\theta}(dx),
\]
where we use \reff{eq:bN-moment}
to get  the equality.
By construction $\cp^\theta_1$
and $\cp^\theta_2$ are independent Poisson point measures. Therefore,
 $\cp^\theta_1+ \cp^\theta_2$ is a Poisson point measure with intensity:
\[
\bm^{\ct_{\theta}}(dx)\left[ \lambda + \psi(\theta+\eta)- \psi(\theta)
  -\psi(\eta) \right]
= \psi_\theta(\eta) \bm^{\ct_{\theta}}(dx).
\]
To conclude, notice that $\tau_\theta(\lambda)$ is the sub-tree generated
by  the marked  leaves before  time $\lambda$  of $\ct_\theta$,  which are
given by the  atoms of $\cp^\theta_1$, and the roots  $x_j$ of the trees
$\ct^j_0$ having marked leaves before  time $\lambda$, that is the atoms
of  $\cp^\theta_2$.  Then  use  that  $\ct_\theta$  under  $\N^\psi$  is
distributed as $\ct_0$ under $\N^{\psi_\theta}$ to conclude.

For $\lambda=0$, we have  $ \cp_1^\theta=0$ and $\ct^j_0$ contributes to
$\tau_\theta(0)$  if and  only  if it  has  infinite mass.   So, in  the
previous  argument, one  has  to replace  $\cp^\theta_2$ by  $\sum_{j\in
  J^{\theta,0}} \delta_{y_j} \ind_{\{\sigma^{\ct^j_0}=+\infty \}}$ which
is a Poisson point measure with intensity:
\[
\bm^{\ct_{\theta}}(dx)\cn_\theta^\psi\left[\sigma=+\infty \right]
=\psi_\theta(\eta)  \,
\bm^{\ct_\theta}(dx).
\]
Hence the conclusion follows.
\end{proof}

\begin{rem}
   \label{rem:k-marks}
   Assume $\lambda>0$.   Using the notation  from the
   previous proof, for $k\in \N^*$, we let:
\[
Y_k=\Card  \{j\in  J^{\theta,0};\;\; \Card
\left(\Lf (\tau_0(\lambda)) \cap \ct^j_0 \right) =k\}
\]
be the number of trees grafted on $\ct_\theta$ having exactly $k$ leaves
marked at  time $\lambda$  and $Y_0=\langle \cp_1^\theta,  \ind \rangle=
\Card  \left(\Lf (\tau_\theta(\lambda))  \cap \Lf  (\ct)\right)$  be the
number of  marked leaves on  $\ct_\theta$. We get that  conditionally on
$\ct_\theta$, the random variables  $(Y_k, k\in \N)$ are independent,
$Y_0$ is  Poisson with parameter $\lambda \sigma_\theta$,  and for $k\in
\N^*$,  $Y_k$ is Poisson  with parameter  $\sigma_\theta \cn_\theta^\psi
\left[(\lambda\sigma)^k \expp{-\lambda \sigma}\right]/k!$.
\end{rem}

Using  the Girsanov transformation  from Section  \ref{sec:Girsanov}, we
will give a Girsanov transformation for $\tau(\lambda)$.

For $\ct\in \T$, let $L(\ct)=\Card \Lf (\ct)$ be the number of leaves of
the   tree   $\ct$   and  \beqlb\label{defL}L(a,\ct)=L(a,\ct^{(a)})=\Card   \{x\in   \ct;
d^\ct(\emptyset,x)=a\}\eeqlb be  the number of elements of  $\ct$ at distance
$a$ from the root.  Note that:
\begin{equation}
   \label{eq:psi-eta-q0}
\eta=\psi^{-1}(\lambda)=q_0+\psi^{-1}_{q_0}(\lambda),\quad\text{and}\quad
\psi'(\eta)=
\psi'(\psi^{-1}(\lambda))=\psi'_{q_0}(\psi_{q_0}^{-1}(\lambda)).
\end{equation}

We first state a preliminary Lemma. Let $\rP^{\psi,\lambda}(d\G)$ denote the
distribution of the Galton-Watson tree $\G(\psi,\lambda)$ defined in
Section \ref{sec:GW}. 

\begin{lemma}
\label{lem:super}
Let $\psi$ be a branching mechanism satisfying (H1-3). Let $\lambda\geq
0$ and $\eta=\psi^{-1}(\lambda)>0$. 
For any non-negative measurable function $F$ on $\T$ and $a\geq0$, we have:
\[
\rE^{\psi,\lambda}[F(\G^{(a)})]=
\rE^{\psi_{q_0},\lambda}\left[F(\G^{(a)})
\left(\frac{\eta}{\eta-q_0}\right)^{L(a,\G)-1}\right].
\]
\end{lemma}

\begin{proof}
  Let  $(p_{(\psi,\lz)}(n),   n\in  \N)$  be   the  probability  measure
  determined  by   $g_{(\psi,\lz)}$  defined  by   \reff{eq:def-g}.  Then  $
  p_{(\psi,\lz)}(1)=0$ and for $n\neq 1$, we have:
\begin{equation}
   \label{eq:p-psi-lambda}
p_{(\psi,\lambda)}(n)
=\frac{g_{(\psi,\lambda)}^{(n)}(0)}{n!}
=\frac{|\psi^{(n)}(\eta)| \, \eta^{n-1}}
{\psi'(\eta)n!}\cdot
\end{equation}
Thanks   to    \reff{eq:psi-eta-q0}, we have
$\psi^{-1}_{q_0}(\lambda)=\eta-q_0$ and for    $n\geq 0$,
$\psi^{(n)}(\eta)           =\psi_{q_0}^{(n)}(\eta-q_0)$.            Set
$u=(\eta-q_0)/\eta$. Then, we have for $n\in \N$:
\begin{multline}
   \label{eq:p_0=up}
p_{(\psi_{q_0},\lambda)}(n)
=\frac{ |\psi_{q_0}^{(n)} (\psi_{q_0}^{-1}(\lambda))|
\,  (\psi^{-1}_{q_0}(\lambda))^{n-1}}{ \psi'_{q_0}
  (\psi_{q_0}^{-1}(\lambda))n!} \,\ind_{\{n\neq 1\}}
\\=u^{n-1} \, \frac{|\psi^{(n)}(\eta)| \, \eta^{n-1}}
{\psi'(\eta)n!} \,\ind_{\{n\neq 1\}}
=u ^{n-1}p_{(\psi,\lambda)}(n).
\end{multline}
The  number  of  leaves  of  $\G^{(a)}$  which are  leaves  of  $\G$  is
$\rP^{\psi_{q_0},\lambda}(d\G)$-a.s.   given,    for   fixed   $a$,   by
$L(\G^{(a)})  -  L(a,  \G)   $.   Thanks  to  \reff{eq:psi-eta-q0},  the
individual      lifetimes       under      $\rP^{\psi,\lambda}$      and
$\rP^{\psi_{q_0},\lambda}$ have the same distribution. Recall $\kappa_x$
is the number of children of $x$. Therefore, we have:
\begin{align*}
\rE^{\psi,\lambda}[F(\G^{(a)})]
&=\rE^{\psi_{q_0},\lambda}\left[
F(\G^{(a)}) \left(\frac{p_{(\psi,\lz)}(0)}
{p_{(\psi_{q_0},\lz)}(0)}\right)^{L(\G^{(a)})- L(a, \G)}
\prod_{x\in \Br (\G^{(a)})}\frac{p_{(\psi,\lz)}(\kappa_{x})}
{p_{(\psi_{q_0},\lz)}(\kappa_{x})}\right]\\
&=\rE^{\psi_{q_0},\lambda}\left[
F(\G^{(a)})
u^{L(\G^{(a)}) - L(a, \G) - \sum_{x\in
    \Br(\G^{(a)})}(\kappa_{x}-1)}\right]\\
&=\rE^{\psi_{q_0},\lambda}\left[
F(\G^{(a)})
u^{1-L(a,\G) }\right],
\end{align*}
where the last equality is a consequence of the following fact for
finite discrete trees $\G$:
\[
1+\sum_{x\in \Br(\G)}(\kappa_{x}-1)=
L(\G).
\]
\end{proof}
Recall (\ref{defL}).
We shall consider the following whole families of sub-trees and leaves:
\[
\tau^{(a)}_{\theta,\lambda}=\{ \tau^{(a)}_\theta(z), z\geq \lambda\},\quad
L(a, \tau_{\theta,\lambda})=L(a, \tau_{\theta,\lambda}^{(a)})=\{L(a, \tau_{\theta}^{(a)}(z)), z\geq\lambda\}.
\]
We have the following Girsanov theorem.

\begin{theo}
\label{thesuper}
Let $\psi$ be a branching mechanism satisfying (H1-3).  Let $\lambda>
0$ and $\eta=\psi^{-1}(\lambda)$. If  $\psi$ is super-critical, then for
any non-negative measurable functional $H$ on the Skorokhod space
$\DD([\lambda, +\infty ), \T)$, we have:
\begin{equation}
   \label{eq:tau-girsanov}
\N^{\psi}\left[H(\tau_{0,\lambda}^{(a)})\ind_{\{M_\lambda\geq
    1\}}\right]=\N^{\psi_{q_0}}\left[ 
\left(\frac{\eta}{\eta- q_0}\right)^{L({a},\tau_0(\lambda))}
H(\tau_{0,\lambda}^{(a)}) \ind_{\{M_\lambda\geq
    1\}}\right].
\end{equation}
\end{theo}

\begin{proof}
  Recall the random measure  $\cn_a^\ct$ defined in \reff{eq:def-cna} is
  according to  the branching property (iv), conditionally  on $\ct^{(a)}$, a
  Poisson  point  measure  with intensity  $\ell^a(dx)\N^\psi[dT]$.   We
  deduce      that,       conditionally      on      $\ct^{(a)}$,      $L(a,
  \tau_0(\lambda))=L(a,\tau_0^{(a)}(\lz))$ is a  Poisson random variable with
  parameter:
\[
\N^\psi\left[1-\expp{-\lambda\sigma}\right]\cz_a= \eta \cz_a.
\]
Let $\cp^\psi$ be, conditionally on $\ct^{(a)}$ and $\tau^{(a)}_0(\lambda)$, a
Poisson point measure on $[\lambda, +\infty )$ with intensity $
\cz_a \left(\psi^{-1}\right)'(z) \, dz$.
We consider the family of random variables:
\[
\cp^\psi_\lambda=\{ \cp^\psi([\lambda,z]), z\geq \lambda\}.
\]
Using again  the branching property  (iv), we get that,  under $\N^\psi$
and conditionally on $\ct^{(a)}$, $L\left(a, \tau_{0,\lambda}\right)$ is
distributed as $L\left(a,  \tau_0 (\lambda) \right) + \cp^\psi_\lambda:=\{ L\left(a,  \tau_0 (\lambda) \right) +\cp^\psi([\lambda,z]), z\geq \lambda\}$.
Then notice that the first equality of \reff{eq:psi-eta-q0} implies that
$\cp^\psi_\lambda$  under $\N^\psi\left[\,\cdot\,  |\ct^{(a)}\right]$ is
distributed          as         $\cp^{\psi_{q_0}}_\lambda$         under
$\N^{\psi_{q_0}}\left[\,\cdot\, |\ct^{(a)} \right]$.  We set:
\[
F(\ct^{(a)},
L(a,\tau_{0,\lambda}^{(a)}))=\N^{\psi}\left[H(\tau_{0,\lambda}^{(a)}) \ind_{\{M_\lambda\geq
    1\}} \, | \, 
\; \ct^{(a)}, L(a,\tau_{0,\lambda}^{(a)})\right].
\]
We deduce that:
\begin{align*}
 \N^{\psi}\left[H(\tau_{0,\lambda}^{(a)})\ind_{\{M_\lambda\geq
    1\}}\right] 
&= \N^{\psi}\left[F(\ct^{(a)},
L(a,\tau_{0,\lambda}^{(a)}))\right]  \\
&= \N^{\psi}\left[F(\ct^{(a)},
L(a,\tau_0^{(a)}(\lambda))+ \cp^\psi_\lambda)\right]  \\
&= \N^{\psi}\left[\sum_{k=0}^\infty F(\ct^{(a)}, k+ \cp^\psi_\lambda)
  \frac{(\eta \cz_a)^k}{k!}\expp{- \eta \cz_a}\right]  \\
&= \N^{\psi_{q_0}}\left[\sum_{k=0}^\infty F(\ct^{(a)}, k+
  \cp^{\psi_{q_0}}_\lambda)
  \frac{(\eta \cz_a)^k}{k!}\expp{- (\eta- q_0) \cz_a}\right] ,
\end{align*}
where   we  used   the  conditional   independence  of   $\cp^\psi$  and
$\tau_0^{(a)}(\lambda)$  given $\ct^{(a)}$ for  the third  equality, the
Girsanov transformation  \reff{eq:GPx} for  the last equality  (and that
$\psi(q_0)=0$).   Using  $\psi_{q_0}^{-1}(\lambda)=\eta-q_0$, we  notice
that $L(a, \tau_0(\lambda))$  is under $\N^{\psi_{q_0}}\left[\, \cdot\,|
  \ct^{(a)} \right]$ a Poisson random variable with parameter:
\[
\N^{\psi_{q_0}}\left[1-\expp{-\lambda\sigma}\right]\cz_a= (\eta-q_0)
\cz_a.
\]
Therefore, we obtain:
\begin{align*}
 \N^{\psi}\left[H(\tau_{0,\lambda}^{(a)})\ind_{\{M_\lambda\geq
    1\}}\right]
&= \N^{\psi_{q_0}}\left[\sum_{k=0}^\infty
  \left(\frac{\eta}{\eta-q_0}\right)^k F(\ct^{(a)}, k+
  \cp^{\psi_{q_0}}_\lambda)
  \frac{((\eta-q_0) \cz_a)^k}{k!}\expp{- (\eta- q_0) \cz_a}\right] \\
&= \N^{\psi_{q_0}}\left[
  \left(\frac{\eta}{\eta-q_0}\right)^{L(a, \tau_0(\lambda))}
  F(\ct^{(a)},L(a, \tau_0(\lambda))+ \cp^{\psi_{q_0}}_\lambda)
  \right] \\
&= \N^{\psi_{q_0}}\left[
  \left(\frac{\eta}{\eta-q_0}\right)^{L(a, \tau_0(\lambda))}
  F(\ct^{(a)},L(a, \tau_{0,\lambda}^{(0)}))
  \right] ,
\end{align*}
where we  used   for the
last equality that  under $\N^{\psi_{q_0}}$ and  conditionally  on $\ct^{(a)}$,
$L\left(a, \tau_{0, \lambda}^{(a)}\right)$ is distributed as $L\left(a,
  \tau_0 (\lambda)  \right)  +  \cp^{\psi_{q_0}}_\lambda$.

By construction,  the  distribution  of
$\tau_{0,\lambda}^{(a)}$    conditionally     on     $\ct^{(a)}$    and
$L(a,\tau_{0,\lambda}^{(a)})$  is   the  same   under   $\N^\psi$  and
$\N^{\psi_\theta}$   for   any   $\theta>0$   and  in   particular   for
$\theta=q_0$. We deduce \reff{eq:tau-girsanov}.
 \end{proof}

We  immediately deduce  the following  Corollary.

\begin{cor}
   \label{cor:tau-law}
Let $\psi$ be a branching mechanism satisfying (H1-3). Let $\lambda>0$
and $\eta=\psi^{-1}(\lambda)$. 
Under $\P^{\psi,\lambda}$, for each $\theta\geq0$,
the sub-tree $\tau_{\theta}(\lz)$ is distributed as the  Galton-Watson real
 tree $\cg(\psi_\theta, \psi_\theta(\eta))$ with  mass
 measure 
 given by \reff{eq:mass-tql}.
\end{cor}
Recall $\rP^{\psi,\lambda}(d\G)$ denotes the
distribution of the Galton-Watson tree $\G(\psi,\lambda)$ defined in
Section \ref{sec:GW}. 

\begin{proof}
If $\psi$ is (sub)critical, then this is a consequence of  Theorem 3.2.1
in \cite{dlg:rtlpsbp} and      Proposition
\ref{prop:marginal}. Now we assume that $\psi$ is super-critical.
Notice that:
\[
\frac{\eta}{\eta-q_0}=\frac{{\N^{\psi}}[M_{\lambda}\geq1]}
{{\N^{\psi_{q_0}}}[M_{\lambda}\geq1]}.
\]
Using Theorem \ref{thesuper},
this gives that
for  $a>0$ and $G$ a non-negative measurable functional defined on $\T$:
\[
 \E^{\psi,\lambda}\left[G(\tau_{0}^{(a)}(\lambda))\right]
= \E^{\psi_{q_0},\lambda}\left[
  \left(\frac{\eta}{\eta-q_0}\right)^{L(a, \tau_0(\lambda))-1}
 G(\tau_{0}^{(a)}(\lambda))  \right].
\]
Recall that if $(\ct,\emptyset,d, \bm)$  is a measured rooted real tree,
then  we denote by  $\tilde \ct$  the real  tree $(\ct,  \emptyset, d)$.
Since  $\psi_{q_0}$   is  sub-critical,  thanks  to   Theorem  3.2.1  in
\cite{dlg:rtlpsbp}, we get that under $\P^{\psi_{q_0},\lambda}$, $\tilde
\tau_0(\lambda)$  has distribution  $\rP^{\psi_{q_0},\lambda}$.  Then by
Lemma  \ref{lem:super}, we get  that under  $\P^{\psi,\lambda}$, $\tilde
\tau_0(\lambda)$  has  distribution   $\rP^{\psi,  \lambda}$.  Then  use
Proposition  \ref{prop:marginal}  to get  that  for each  $\theta\geq0$,
$\tilde   \tau  _{\theta}   (\lambda)$  under   $\P^{\psi,\lambda}$  has
distribution $\rP^{\psi_{\theta}, \psi_{\theta}(\eta)}$.
\end{proof}

The following Corollary is another direct consequence of Theorem
\ref{thesuper}.
\begin{cor}
\label{cor:mart-leaves}
Let  $\lambda>0$   and  $a>0$  be  fixed.   Under  $\N^{\psi_{q_0}}$  on
$\{M_\lambda\geq 1\}$, the process $(Q_z, z\geq \lambda)$ defined by:
\[
Q_z=\left(\frac{\psi^{-1}(z)}{\psi^{-1}(z) -q_0}\right)^{L(a, \tau_0(z))}
\]
is a backward martingale with respect to the filtration $(\cq_z, z\geq
\lambda)$ with $\cq_z=\sigma(\tau_0({z'}); z'\geq z).$
\end{cor}




We present an other  Girsanov transformation for sub-trees.

\begin{remark}
\label{prop:qq-Girsanov}
Let $\psi$ be a branching mechanism satisfying (H1-3).
For any $q\geq\theta \geq0$,  $a>0$ and $F$ a non-negative measurable
functional,  we have:
\begin{equation}
   \label{eq:qq-Girsanov}
\E^{\psi,\lambda} \left[F(\tilde \tau_q^{(a)}(\lambda))\right]
=\E^{\psi,\lambda} \left[F(\tilde \tau_\theta^{(a)}(\lambda))
  N^{\theta,q}_{a,\lambda}(\tau_\theta(\lambda)) \right],
\end{equation}
where $N^{\theta,q}_\lambda$ is defined for  discrete trees  by:
\[
N_{a,\lambda}^{ \theta, q}(T)
=\left(\frac{\psi_q(\eta)}{\psi_{\theta}(\eta)}\right)^{L(T^{(a)})
  -L(a,T)}
\expp{(\psi'(\theta+\eta)-\psi'(q+\eta)) \ell^T(T^{(a)})}
\prod_{x\in \Br(T^{(a)})}
\frac{\psi_q^{(\kappa_x)}(\eta)}{\psi_{\theta}^{(\kappa_x)}(\eta)},
\]
with      the      convention     $\prod_{x\in\emptyset}=1$.       Under
$\P^{\psi,\lambda}$,                     the                     process
$N^{\theta,q}_\lambda=\left(N^{\theta,q}_{a,\lambda}(\tilde
  \tau_\theta(\lambda)) ,  a\geq 0\right)$ is a  martingale with respect
to   the  filtration   $\left(\sigma(\tilde \tau_\theta^{(a)}(\lambda)),  a\geq
  0\right)$.
\end{remark}

\section{Convergence of the sub-tree processes}
\label{SecsubCov}

We provide  an alternative proof of  the convergence of  the sub-trees to
the  L\'{e}vy tree  from \cite{dw:glt}  using  the Gromov-Hausdorff-Prohorov
distance  on $\T$ which  relies on  the Girsanov  transformation. Recall
that for simplicity, we  identify $T$ and $(T,d^T,\emptyset^T, \bm^T)\in
\T$.  And,   under  $\P^\psi_r$  or  $\N^\psi$,  the   mass  measure  on
$\tau_0(\lambda)$ is given by \reff{eq:mass-tql}.

\begin{theo}
   \label{theo:cv-dGHPc}
Let  $\psi$ be  a branching mechanism  satisfying (H1-3).
We have $\N^\psi$-a.e. or $\P^\psi_r$-a.s.:
\begin{equation}
   \label{eq:cv-dGHP}
\lim_{\lambda\rightarrow+\infty } d_{\text{GHP}}(\ct,
\tau_0(\lambda))=0.
\end{equation}
\end{theo}

\begin{proof}
  Under $\N^\psi$, the convergence \reff{eq:cv-dGHP} is a consequence of
  Lemma    \ref{lem:cv-dGHPc}  below  (see    also   Proposition    2.8    in
  \cite{adh:ghptwms}  to get the  $d_{\text{GHP}}$ convergence  from the
  $d_{\text{GHP}}^c$ convergence)  for the (sub)critical  case and Lemma
  \ref{lem:cv-dGHP} below   for   the    super-critical    case.   Then    the
  $\P_r^\psi$-a.s.  convergence  is a consequence  of the representation
  of $\P^\psi_r$ from Section \ref{sec:meas-LT}.
\end{proof}

\begin{rem}
   \label{rem:decomp-super}
   Notice  in particular that  Theorem \ref{theo:cv-dGHPc}  asserts that
   $(\cf,(\cf(\lambda),  \lambda\geq 0)) $  in \cite{dw:glt}  and $(\ct,
   (\tau_0(\lambda),  \lambda\geq 0))$ have  the same  distribution.  In
   particular,  this implies  that the  distribution  for super-critical
   L\'{e}vy trees  defined in \cite{dw:glt}  based on a  coloring leaves
   process and the one defined in \cite{adh:etiltvp} based on a Girsanov
   transformation     are      the     same.      Therefore,     Theorem
   \ref{theo:decomp-tree} is also valid for $\psi$ super-critical.
\end{rem}

\begin{rem}
  \label{rem:cv-proc}   The  pruning   sub-tree   process
  $(\tau_\theta(\lambda),  \theta\geq 0)$  is  a piece-wise  c\`adl\`ag $\T$-valued
  process. It is easy to check, using the representation of the backward
  process   in  \cite{adh:etiltvp},  that   the  pruning   tree  process
  $(\ct_\theta, \theta\geq  0)$ is  also a c\`adl\`ag $\T$-valued  process. Then,
  one can  also prove  the convergence, with  respect to  the Skorokhod
  topology,    of    the    pruning    sub-tree    process
  $(\tau_\theta(\lambda),  \theta\geq 0)$  towards  the  pruning  tree process
  $(\ct_\theta, \theta\geq 0)$ as $\lambda$ goes to infinity.
\end{rem}

Lemma  \ref{lem:cv-dGHPc}  is stated  in  Section \ref{sec:GHP-sub}  and
Lemma   \ref{lem:cv-dGHP}  in   Section   \ref{sec:GHP-super}.   Section
\ref{sec:d-subtree} presents preliminaries  on approximation of trees by
discrete sub-trees.

\subsection{Distance between trees and discrete sub-trees}
\label{sec:d-subtree}
In this Section, we present an immediate convergence result from
sub-trees to trees for trees coded by a function.

Let $f$ be  a non-negative continuous function
with compact support  s.t. $f(0)=0$. We set $\sigma=\sup\{t;
f(t)>0\}$. We define:
\[
d^f(x,y)=f(x)+f(y) - 2 \inf_{u\in [x\wedge y, x\vee y]} f(u)
\]
and the equivalence relation: $x\sim y$ if $d^f(x,y)=0$. We set
$T^f=[0,\sigma]/\sim$. Let $p^f$ be the projection from
$[0,\sigma]$ to $T^f$, with $p^f(x)$ the equivalent class of $x$ in
$T^f$. Let $\bm^f$ be the image of the Lebesgue measure on $[0,\sigma ]$
by the projection $p^f$. Set $\emptyset^f=p^f(0)$ and we still denote by
$d^f$ the distance on $T$, image of $d^f$ by $p^f$. It is well known
that $(T^f,d^f, \emptyset^f, \bm^f)$ is a measured rooted compact real
tree.

Let  $\Delta=\{y_0, \ldots, y_{N_\Delta}\}$, with $1\leq
N_\Delta<+\infty $  and $0=y_0<\cdots  < y_{N_\Delta}\leq \sigma$,  be a
finite   subdivision   of   $[0,\sigma]$.    Let   $|\Delta|=\sup_{0\leq
  i<N_\Delta} y_{i+1} -y_i$  be the mesh of the  subdivision. For $0\leq
i<N_\Delta$, let  $\bar y_i\in [y_i,  y_{i+1}]$ such that  $f(\bar y_i)=
\inf_{u\in  [y_i, y_{i+1}]}  f(u)$.  We consider  $f_\Delta$ the  linear
interpolation of  the points $\{(y_i, f(y_i)), (\bar  y_i, f(\bar y_i));
0\leq   i<N_\Delta\}  \cup  \{(y_{N_\Delta},   f(y_{N_\Delta}))\}$.   By
construction  $T^{f_\Delta}$   is   the  smallest   sub-tree  of  $T^f$
containing $\{ p^f(y_i), 0\leq i\leq N_\Delta\}$.

Let $a_\Delta\geq 0$  and $ \bm^{f,\Delta}$ be the  image of the measure
$\mu_\Delta=a_\Delta  \sum_{y\in  \Delta, y\neq  0}  \delta_{y}$ by  the
projection   $p^f$.  We   consider   the  measured   rooted  real   tree
$T^{f,\Delta}=(T^{f_\Delta},          d^{f_\Delta},         \emptyset^f,
\bm^{f,\Delta})$. It is elementary to get:
\begin{equation}
   \label{eq:distTf-Tfdiscret}
 d_{\text{GHP}}^c(T^f, T^{f,\Delta}) \leq  \sup_{|x-y|\leq  |\Delta|}
 |f(x) -f(y)| +  d_\text{P}^{[0,\sigma]}(\text{Leb}, \mu_\Delta),
\end{equation}
where Leb is the Lebesgue measure on $[0,\sigma]$, and the space
$[0,\sigma]$ is endowed with the usual distance.

\subsection{The (sub)critical case}
\label{sec:GHP-sub}
The main result of this Section is the following Lemma.

\begin{lem}
   \label{lem:cv-dGHPc}
Let  $\psi$ be  a (sub)critical  branching mechanism  satisfying (H1-3).
We have $\N^\psi$-a.e. for all $a_0\geq 0$:
\begin{equation}
   \label{eq:cv-dGHPc}
\lim_{\lambda\rightarrow+\infty }d_{\text{GHP}}^c(\ct,
\tau_0(\lambda))=0
\quad\text{and}\quad
\lim_{\lambda\rightarrow+\infty } \sup_{a\leq a_0} d_{\text{GHP}}^c(\ct^{(a)},
\tau_0^{(a)}(\lambda)) =0.
\end{equation}
\end{lem}

\begin{proof}
   According  to \cite{dlg:rtlpsbp}, there  exists a  continuous stochastic
process  $h$,  called the height  process,  such  that  under its  excursion
measure it  has compact  support $[0,\sigma^h]$ and  $(T^h,\sigma^h)$ is
distributed at $(\ct, \sigma)$  under $\N^\psi$.  Notice that
the continuity of the height  process is a consequence of (H2).
Conditionally
on $h$, let $\cp=\sum_{i \in  I} \delta_{(y_i, t_i)}$ be a Poisson point
measure on $[0, \sigma]\times \R_+$ with intensity $dydt$. For
$\lambda>0$, we set:
\[
\Delta_\lambda=\{y_i; i\in I \text{ and } t_i\leq \lambda\}\cup\{0\}
\quad\text{and}\quad
\mu_{\Delta_\lambda}=\inv{\lambda} \sum_{y\in \Delta_\lambda, y\neq 0} \delta_y.
\]

By construction, we get the following equality in distribution:
\[
(T^h, (T^{h,\Delta_\lambda}, \lambda\geq 0))
\stackrel{\text{(d)}}{=}
(\ct, (\tau_0(\lambda), \lambda\geq 0)).
\]

The properties of the Poisson point measures imply  that a.e. under the
excursion measure of $h$,   $\lim_{\lambda\rightarrow+\infty }
|\Delta_\lambda|=0$ and
$\lim_{\lambda\rightarrow+\infty } d_\text{P}^{[0,\sigma^h]}(\text{Leb},
\mu_{\Delta_\lambda})=0$. Thus, we deduce from Section \ref{sec:d-subtree} and
\reff{eq:distTf-Tfdiscret} that a.e. under the excursion measure of $h$,
\[
\lim_{\lambda\rightarrow+\infty } d_{\text{GHP}}^c(T^h,
T^{h,\Delta_\lambda}) =0.
\]
Thus, we obtain the first part
of \reff{eq:cv-dGHPc}.

We set $\varepsilon_\lambda=d_{\text{GHP}}^c(\ct,
\tau_0(\lambda))$.
According to the proof of Proposition 2.8 in \cite{adh:ghptwms}, we
have, for $a\geq 0$:
\[
d_{\text{GHP}}^c(\ct^{(a)},
\tau_0^{(a)}(\lambda)) \leq  3\varepsilon_\lambda  +
\bm^\ct\left(\ct^{(a+2\varepsilon_\lambda)}\backslash
\ct^{(a-\varepsilon_\lambda)}\right).
\]
Using \reff{eq:int-la} and the definition of $\cz$, we deduce that for
$a_0\geq 0$:
\[
\sup_{a\leq a_0} d_{\text{GHP}}^c(\ct^{(a)},
\tau_0^{(a)}(\lambda)) \leq  3\left(1  +  \sup_{a\leq
    a_0+2\varepsilon_\lambda} \cz_a\right)\varepsilon_\lambda.
\]
We deduce then the second part of \reff{eq:cv-dGHPc} from the first part
of \reff{eq:cv-dGHPc}. This ends the proof of the Lemma.
\end{proof}

\subsection{The super-critical case}
\label{sec:GHP-super}
The main result of this Section is the following Lemma.

\begin{lem}
   \label{lem:cv-dGHP}
Let  $\psi$ be  a super-critical  branching mechanism  satisfying (H1-3).
We have $\N^\psi$-a.e.:
\[
\lim_{\lambda\rightarrow+\infty } d_{\text{GHP}}(\ct,
\tau_0(\lambda))=0.
\]
\end{lem}

\begin{proof}
We deduce
from Theorem   \ref{thesuper}
that for  $a>0$:
\begin{multline*}
\N^{\psi}\left[\liminf_{\lambda\rightarrow+\infty }
\int_0^a  \expp{-r} \left(1 \wedge
d^c_{\text{GHP}}\left(\ct^{(r)}, \tau_0^{(r)}(\lambda) \right)
\right) \ dr>0\right]\\
= \N^{\psi_{q_0}}\left[\left(\frac{\psi^{-1}(1)}{\psi^{-1}(1)-
      q_0}\right)^{L({a},\tau_0(1))} , \,
\liminf_{\lambda\rightarrow+\infty}
\int_0^a  \expp{-r} \left(1 \wedge
d^c_{\text{GHP}}\left(\ct^{(r)}, \tau_0^{(r)}(\lambda) \right)
\right) \ dr>0\right].
\end{multline*}
Then use \reff{eq:cv-dGHPc} to get that the right hand-side in the
previous equality is 0 for all $a>0$. This implies that $
   \N^{\psi}\left[\liminf_{\lambda\rightarrow+\infty } d_{\text{GHP}}(\ct,
\tau_0(\lambda))>0 \right]=0$.
\end{proof}

\section{Pruning and growth of the discrete sub-trees}
\label{sec:subgrowth}

\subsection{The pruning process}

We define the following pruning procedure for the discrete sub-trees.   Under
$\P^{\psi,\lambda}$, let $\cT$ be distributed as $\tau_0(\lambda)$.
Conditionally   on  $\cT$,   we   consider  a   Poisson  point   measure
$\cm^\Sk(d\theta, dy)$ on $\R_+\times \cT$ with intensity:
\[
\psi''(\eta+\theta) \ind_{[0,+\infty)}(\theta)d\theta\,
\ell^\cT(dy)
\]
and an independent family  of independent random variables $(\xi_x, x\in
\Br(\cT))$, such that the distribution of $\xi_x$  has density:
\[
-\frac{\psi^{(\kappa_x+1)}(\eta+z)}{\psi^{(\kappa_x)}(\eta)} \, \ind_{\{z>0\}} dz.
\]
We define the mark process:
\[
\cm^\cT(d\theta,dy)= \cm^\Sk(d\theta, dy) + \sum_{x\in \Br(\cT)}
 \delta_{(\xi_x,x)} (d\theta, dy).
\]
For every $x\in\cT$, we consider the corresponding record process on $\cT$:
\[
\theta^\cT(x)=\inf\{\theta>0,\ \cm^\cT([0,\theta]\times \lb\emptyset,x\rb)>0\}.
\]
We define the pruned tree at time $q\geq 0$ as:
\[
\cT_q=\{x\in\cT,\ \theta^\cT(x)\ge q\}
\]
with  the  induced  metric,   root  $\emptyset$  and  mass  measure
$
\bm^{\cT_q}=\inv{\psi_q(\eta)} \sum_{x\in \Lf(\cT_q)} \delta_x$.
Then we have the following theorem.
\begin{theo}
\label{theo:pruning}
Let $\psi$ be a  branching mechanism satisfying (H1-3). Let $\lambda\geq
0$     such     that     $\eta=\psi^{-1}(\lambda)>0$.     Then     under
$\P^{\psi,\lambda}$,    the   two    processes   $(\tau_\theta(\lambda),
\theta\geq  0)$   and  $(\cT_\theta,   \theta\geq  0)$  have   the  same
distribution.
\end{theo}

\begin{proof}
  The  proof  is  based  on Theorem  \ref{theo:decomp-tree}  and  Remark
  \ref{rem:decomp-super}.  Notice that the  processes $(\tau_\theta(\lambda),
  \theta\geq  0)$ and  $(\cT_\theta, \theta\geq  0)$ are  by construction
  Markov and right continuous. Therefore, it is enough to check the
  two-dimensional marginals have the same distribution.

  Let $\theta\geq  q\geq 0$.
Recall  the pruning  procedure defined  in Section
\ref{sec:pruning}. On  one hand, a
mark appears on the  skeleton of $\tau_q(\lambda)$ before time $\theta$,
if this is a mark which appears before time $\theta$ and which is either
on the skeleton of $\ct_q$ or on a branching point of $\ct_q$. Those marks
which  initially are  on  the  skeleton of  $\ct_q$  are distributed  on
$\tau_q(\lambda)$   according  to   a  Poisson   point  measure   with  intensity
$2\beta(\theta-q) \ell^{\tau_q(\lambda)}(dy)  $. A node  of $\ct_q$ with
mass   $r$  has   a   mark  before   time   $\theta$  with   probability
$1-\expp{-(\theta-q) r}$. And the nodes  of $\ct_q$ with mass $r$ which lies
on  the   skeleton  of  $\tau_q(\lambda)   $  are,  thanks   to  Theorem
\ref{theo:decomp-tree}, distributed  on $\tau_q(\lambda)$ according  to a Poisson
point   measure   with   intensity   $  r\expp{-r\eta   -rq}   \pi(dr)\,
\ell^{\tau_q(\lambda)}  (dy)$.  This  implies  that  the  marks  on  the
skeleton  of  $\tau_q(\lambda)$  before  time $\theta$  are  distributed
according to a Poisson point measure with intensity:
\begin{align*}
 \left[2\beta(\theta-q) +
\int_{(0,+\infty )} (1-\expp{-(\theta-q) r} )r\expp{-r\eta -rq} \pi(dr)
\right]\, \ell^{\tau_q(\lambda)}(dy)
&= \left[\psi_{\theta}' (\eta)-\psi_q' (\eta)\right] \,
\ell^{\tau_q(\lambda)}(dy) \\
&= \int_q^\theta \psi''(\eta+z) dz \,
\ell^{\tau_q(\lambda)}(dy) .
\end{align*}

On the other hand, if $x$  is a node of $\tau_q(\lambda)$ with number of
children $\kappa_x$, then a mark  appears on it before time $\theta$, if
it  appears before time  $\theta$ on  $\ct_q$. According  to Proposition
\ref{prop:marginal},  $(\ct_q,\tau_q(\psi_q(\eta)))$  is distributed  as
$(\ct_0,  \tau_0(\lambda))$ under $\N^{\psi_q}$.   We deduce,  thanks to
Theorem   \ref{theo:decomp-tree},   that    the   mass   $\Delta_x$   is
conditionally    on    $\tau_0(\lambda)$    distributed   according    to
$\Gamma^{\psi_q}_{\kappa_x,      \psi_q(\eta)}     $      defined     by
\reff{eq:def-Gamma}.  Therefore  a  mark  appears  on the  node  $x$  of
$\tau_q(\lambda)$ before time $\theta$ with probability:
\[
 \int \Gamma_{\kappa_x,\psi_q(\eta)}^{\psi_q}(dr)(1-\expp{-(\theta-q)r})=
 1-\frac{\psi_{\theta}^{(\kappa_x)} (\eta)}{\psi_q^{(\kappa_x)} (\eta)}
=\P(\xi_x<\theta|\xi_x>q).
\]
By  construction  of  $\cT_\theta$  from  $\cT_q$, we  deduce  that  the
distribution of $\cT_\theta$ conditionally  on $\{\cT_q=T\}$ is the same
as   the  distribution   of  $\tau_\theta(\lambda)$   conditionally  on
$\{\tau_q(\lambda)=T\}$.    Then    use   that   $\cT_0$    is   distributed   as
$\tau_0(\lambda)$, to deduce that $\cT_\theta$ has the same distribution
as  $\tau_\theta(\lambda)$.  Thus,  we    get  that  the  processes
$(\tau_\theta(\lambda), \theta\geq 0)$  and $(\cT_\theta, \theta\geq 0)$
have the same two-dimensional marginals distribution.
\end{proof}

\begin{rem}
   \label{rem:shift}
By construction of $\cT$ and thanks to Proposition \ref{prop:marginal},
we get that $(\cT_{\theta+q}, \theta\geq 0)$ under $\P^{\psi,\lambda}$ is
distributed as $(\cT_\theta, \theta\geq 0)$ under $\P^{\psi_q,
  \psi_q(\eta)}$. 
\end{rem}

\subsection{The growth process}
Let $\lambda>0$.  Theorem \ref{theo:pruning} gives the pruning procedure
of  the  sub-tree  process.   Conversely,  we will  also  give  a  growth
procedure for the time reversed  sub-tree process. However,  if $\ct_\theta$
can be defined on $\Theta^\psi$ simultaneously, this is no more the case
for  $\tau_\theta(\lambda)$.  Recall $\psi_\theta(\eta)\geq  \lambda> 0$
for $\theta\geq 0$.  We define:
\begin{equation}
   \label{eq:def-q-l}
\theta_\lambda=\inf\{\theta\in \Theta^\psi; \psi_\theta(\eta)\geq 0\}
\quad\text{and}\quad
\Theta^{\psi, \lambda}=[\theta_\lambda, +\infty ) \cap \Theta^\psi.
\end{equation}
Notice that $\theta_\lambda\leq 0$. Theorem \ref{theo:pruning},
Remark \ref{rem:shift} and the
Kolmogorov extension theorem insure that there exists a process
$(\tau_\theta(\lambda), \theta\in \Theta^{\psi,\lambda})$ under
$\P^{\psi,\lambda}$, such that for all $q\in \Theta^{\psi,\lambda}$ the process
$(\tau_{\theta+q}(\lambda),  \theta\geq 0)$ is distributed as
$(\tau_{\theta}(\psi_q(\eta)), \theta\geq 0)$ under $\P^{\psi_q,\psi_q(\eta)}$.

We consider the  function
$g_{(\psi, \lambda)}^{q,\theta}$ defined for $q\in
\Theta^{\psi,\lambda}$ and $\theta> q$ by:
\begin{equation}
   \label{eq:g-psil-qq}
g_{(\psi, \lambda)}^{q,\theta}(r)
=1-\frac{\psi_\theta(\eta(1-r))-
  \psi_q(\eta(1-r))}{ \psi_\theta(\eta)}\cdot
\end{equation}
Notice that $\psi_q(\eta)>0$ and thus $g_{(\psi,\lambda)}^{q,\theta}(1)=1$, $
g_{(\psi,\lambda)}^{q,\theta}(0)=\psi_q(\eta)/\psi_{\theta}(\eta)$
and for $k\in \N^*$:
\[
\left(g_{(\psi,
  \lambda)}^{q,\theta}\right)^{(k)}\!\!\!(0)
=
\frac{(-1)^{k+1}\eta^{k}}{\psi_\theta(\eta)}(\psi^{(k)}
(\theta+\eta)-\psi^{(k)}(q+\eta))  \geq 0.
\]
Since  $\psi$ is  analytical  at least  on  $(\theta_\lambda,+\infty )$,  we
deduce  that  $g_{(\psi,   \lambda)}^{q,\theta}(r)$  is  the  generating
function of a  random variable $K$ taking values  in $\N$. Let $(\tau^k,
k\in \N^*)$ be independents random trees distributed as $\tau_q(\lambda)$
under $\P^{\psi,\lambda}$ and independent of $K$.  We set:
\[
{\G}_{q,\theta}(\psi, \lambda)= \emptyset \circledast_{1\leq k\leq
  K}(\tau^k,\emptyset) ,
\]
with the convention
that $\emptyset \circledast_{1\leq k\leq K}
(\tau^k,\emptyset)=\emptyset$ if $K=0$.

\begin{theo}
   \label{theo:growth}
   Let  $\psi$  be  a  branching  mechanism satisfying  (H1-3).   Let
   $\lambda>0$ and $\eta=\psi^{-1}(\lambda)$. Let  $
   \theta>  q$  with   $q\in  \Theta^{\psi,  \lambda}$.   Then  under
   $\P^{\psi,\lambda} $, conditionally  on $\tau_\theta(\lambda)$,
   $\tau_q(\lambda)$ is distributed as
\[
\tau_\theta(\lambda) \circledast_{x\in \Lf(\tau_\theta(\lambda))}
(\cg_q^x, x),
\]
with mass measure given by \reff{eq:mass-tql} (with $\theta$ replaced by
$q$) and 
where $(\cg_q^x,x\in \Lf(\tau_\theta(\lambda)))$ are independent and
distributed according to ${\G}_{q,\theta}(\psi, \lambda)$.
\end{theo}

We first state a preliminary Lemma.

\begin{lem}
   \label{lem:cn-tl}
Under the Hypothesis of Theorem \ref{theo:growth}, the sub-tree
$\tilde \tau_0(\psi_q(\eta))$ 
is distributed under  $\cn_{\theta-q}^{\psi_q} [\,  \cdot\, |
M_{\psi_q(\eta)}\geq 1]$ as $\cg_{q,\theta}(\psi,\lambda)$ conditionally on
$\cg_{q,\theta}(\psi,\lambda)\neq \emptyset$.
\end{lem}

\begin{proof}
  By construction  of $\cg_{q,\theta}(\psi,\lambda)$, the  Lemma will be
  proved  as  soon  as  we  check   that  the  degree  of  the  root  of
  $\tau_0(\psi_q(\eta))$  under  $\cn_{\theta-q}^{\psi_q}\left[\, \cdot\, |\,
    M_{\psi_q(\eta)}\geq  1 \right]$  is  distributed as  $K$ conditionally  on
  $\{K\geq 1\}$.

Without loss of generality, we may write $\psi$, $\lambda$ and $\theta$ for
$\psi_q$, $\psi_q(\eta)$  and $\theta-q$, that is assume $q=0$.
Let $N_\emptyset$ be the degree of the root $\emptyset$ in
$\tau_0(\lambda)$. Notice that $\{M_\lambda\geq 1\}=\{N_\emptyset\geq
1\}$. We set $h(u)=\cn_{\theta}^{\psi}\left[u^{N_\emptyset}\ind_ {\{N_\emptyset\geq
       1\}}\right]$. 
Notice that, under $\N^{\psi}$, $N_\emptyset$
is 0 or 1 and that, under $\P^{\psi}_r$,
$N_\emptyset$ is a Poisson random variable with mean
$r\N^{\psi}\left[M_\lambda\geq 1 \right]=r \eta$. We deduce
that for $u\in [0,1]$:
\begin{align*}
 h(u)
&=2 \beta \theta u \N^{\psi} \left[M_\lambda\geq 1 \right]
+ \int_{(0,+\infty )} \pi(dr) (1-\expp{-\theta r}) \E_r^\psi
\left[u^{N_\emptyset} \ind_{\{N_\emptyset\geq 1\}}\right]\\
&=2 \beta \theta \eta u 
+ \int_{(0,+\infty )} \pi(dr) (1-\expp{-\theta r}) (\expp{-r\eta (1-u)}
-\expp{-r\eta} ).
\end{align*}
Let $g_0=g^{0,\theta}_{(\psi,\lambda)}$ be the generating function of
$K$ and $g_1$ be the generating function of
$K$ conditionally on $ \{K\geq 1\}$. Elementary computations yields
$g_0(u)=g_0(0)+ h(u)/\psi_\theta(\eta)$. 
We deduce that $g_1(u)=h(u)/h(1)$. This readily implies that
$N_\emptyset$   under  $\cn_{\theta}^{\psi}\left[\, \cdot\, |\,
    M_\lambda\geq  1 \right]$  is  distributed as  $K$ conditionally  on
  $\{K\geq 1\}$.
\end{proof}

\begin{proof}[Proof of Theorem \ref{theo:growth}]
  The   proof   is   very   similar   to  the   proof   of   Proposition
  \ref{prop:marginal}.   From  the   special  Markov   property  Theorem
  \ref{theo:markov-special}, we get:
\[
\ct_q=\ct_{\theta}\circledast_{j\in J^{\theta,q}}(\ct^j_q, x_j),
\]
where $\sum_{j\in J^{\theta,q}}{\dz}_{(x_j, \ct^j_q)}$ is, conditionally
on $\ct_\theta$, a Poisson  point measure on $\ct_{\theta}\times\T$ with
intensity $\bm^{\ct_{\theta}}(dx)\cn_{\theta-q}^{\psi_q}[d\ct]$.  Notice
that  $\ct^j_q$  gives  a  contribution to  $\tau_q(\lambda)$  (that  is
$\ct^j_q\cap \tau_q(\lambda)\neq  \emptyset$) if  there is at  least one
marked leaf on  $\ct^j_q$. Furthermore, if there is  a contribution, then
$\ct^j_q\cap  \tau_q(\lambda)$ is distributed  as $\tau_0(\psi_q(\eta))$
under  $\cn_{\theta-q}^{\psi_q} [\,  \cdot\, |  M_{\psi_q(\eta)}\geq 1]$
but for  the root which  is $x_j$. This  distribution is given  in Lemma
\ref{lem:cn-tl}.  Thanks to \reff{eq:bN-moment}, we have:
\begin{equation}
   \label{eq:cn-M>=1}
\cn_{\theta-q}^{\psi_q} [ M_{\psi_q(\eta)}\geq 1]=
\cn_{\theta-q}^{\psi_q} [1- \expp{-\psi_q(\eta) \sigma}]=\psi(\theta+\eta)-
\psi(\theta) -\psi_q(\eta)=\psi_\theta(\eta) -\psi_q(\eta) .
\end{equation}

Standard results on marked Poisson point process imply that the point
measure on the leaves of $\tau_q(\lambda)$ which are still in
$\tau_\theta(\lambda)$, that is $\sum_{x\in
  \Lf(\tau_\theta(\lambda))\cap \Lf(\tau_q(\lambda))} \delta_x(dy)$,  is,
conditionally on $\ct_\theta$, a Poisson point process on
$\ct_\theta$ with intensity $\psi_q(\eta) \bm^{\ct_\theta}(dy)$, and
is also independent of $\sum_{j\in J^{\theta,q}}{\dz}_{(x_j,
  \ct^j_q)}$.

Using standard results on marked Poisson point measure, we get that
$\tau_q(\lambda)$ can be recovered from $\tau_\theta(\lambda)$ by grafting
independently on
each leaf $x\in \Lf(\tau_\theta(\lambda))$:
\begin{itemize}
   \item Nothing with probability $\psi_q(\eta)/\psi_\theta(\eta)$.
\item A sub-tree distributed as $\cg_{q,\theta}(\psi,\lambda)$ conditionally on
$\cg_{q,\theta}(\psi,\lambda)\neq \emptyset$ with probability
$1-\psi_q(\eta)/\psi_\theta(\eta)$.
 \end{itemize}
Then use that 
$\rP(\cg_{q,\theta}(\psi,\lambda)=\emptyset)=\rP(K=0)=\psi_q(\eta)
/\psi_\theta(\eta)$ and that the mass measure  of $\tau_q(\lambda)$ is
given by \reff{eq:mass-tql} (with $\theta$
replaced by $q$) to end the proof.
\end{proof}

\begin{rem}
   \label{rem:generateur}
We deduce from Theorem  \ref{theo:growth} that the transition rate (for
the backward process) at
time $\theta$ from $\tau_\theta(\lambda)$ to $\tau_\theta(\lambda) \circledast_{1\leq k\leq
  k_0}(\tau^k,x) $, with $x$ a leaf of $\tau_\theta(\lambda)$, is given by:
\[
\frac{(-1)^{k_0+1}\eta^{k_0}}{k_0!}\, 
\frac{\psi^{(k_0+1)}(\theta+\eta)}{\psi_\theta(\eta)} \mu_\theta(d\tau^1)
\cdots \mu_\theta(d\tau^{k_0}),
\]
with $\mu_\theta$ the distribution of $\tau_\theta(\lambda)$ under
$\P^{\psi,\lambda} $. The mass measure process is always defined by
\reff{eq:mass-tql}. 
\end{rem}

\section{Study of Leaves}
\label{sec:leaves}
We first present a martingale based on the total mass of the pruned
process.

\begin{prop}
\label{prop:mart-s}
Let $\psi$ be a branching mechanism satisfying (H1-3). Then under
$\P_r^{\psi}$ and $\N^\psi$, the process $(R_\theta, \theta>q_0)$, with:
\[
R_\theta=\psi'(\theta)\sigma_\theta,
\]
is a backward martingale  with respect to the filtration
$(\cf_{\theta}, \theta>q_0)$  where
$\cf_\theta=\sigma(\ct_q, q\geq \theta)$.
\end{prop}

\begin{proof}
   Let $q_0<q\leq \theta$. According to the special Markov property, we
   have:
\[
(\ct_{\theta}, \sigma_q)\overset{(d)}{=}(\ct_{\theta},
\sigma_{\theta}+\sum_{i\in I}\sigma^{\ct^i}),
\]
where $\sum_{i\in I}\delta_{  \ct^{i}}$ is conditionally on $\ct_\theta$
 a Poisson point  measure on $\T$ with intensity:
\[
\bm^{\ct_\theta}(dx) \cn^{\psi_q}_{\theta-q}[d\ct].
\]
Using \reff{eq:cn-q}
and \reff{eq:bN-moment}, we have:
\[
   \E_r^{\psi}[\sigma_q|\F_{\theta}]
=\E_r^{\psi}\left[\sigma_{\theta}+\sum_{i\in
I}\sigma^{\ct^i} \bigg{|}\F_{\theta}\right]
=\sigma_{\theta}+\sigma_{\theta} \cn^{\psi_q}_{\theta-q}
\left[\sigma\right]
=\frac{\psi'(\theta)}{\psi'(q)}\sigma_{\theta}.
\]
This gives the result under $\P^\psi_r$. The proof is similar under
$\N^\psi$.
\end{proof}

Notice that  Proposition \ref{prop:mart-s} is also  a direct consequence
of the infinitesimal transitions of time-reversed process $(\ct_\theta, \theta\in
\Theta^\psi)$ given in \cite{adh:etiltvp}.

Now we present a result on the number of leaves for the sub-tree
process. Let $\lambda\geq 0$.
 We consider
the leaves
process of the sub-trees $L(\lambda)=\{L_\theta(\lambda), \theta\in
\Theta^{\psi,\lambda}\}$: 
\[
L_\theta(\lambda)=L(\tau_\theta(\lambda))=\Card (\Lf(\tau_\theta(\lambda))).
\]

\begin{prop}
   \label{prop:mart-l}
   Let  $\psi$   be  a  branching  mechanism   satisfying  (H1-3).   Let
   $\lambda\geq      0$     and     $\eta=\psi^{-1}(\lambda)>0$.      Under
   $\P^{\psi,\lambda}$,    the    process   $(R_\theta(\lambda),
   \theta>q_0)$ with:
\[
R_\theta(\lambda)= \frac{\psi'(\theta)} {\psi_\theta(\eta)} L_\theta(\lambda),
\]
is  a  backward  martingale  with
   respect  to   the  filtration  $(\ch_{\theta},   \theta>q_0)$,  where
   $\ch_\theta =\sigma (\tau_q(\lambda), q\geq \theta)$.
\end{prop}

\begin{rem}
  \label{rem:martL} Notice  that $ L_\theta(\lambda)/ \psi_\theta(\eta)$
  is  the total  mass of  $\bm^{\tau_\theta(\lambda)}$.  The convergence
  from  Remark  \ref{rem:cv-proc}  and  the fact  that  $\ct_\theta$  is
  compact for $\theta>q_0$,  implies that (as a process)  the total mass
  of  $\bm^{\tau_\theta(\lambda)}$  converges   to  the  total  mass  of
  $\bm^{\ct_\theta}$  that  is  $\sigma_\theta$  as  $\lambda$  goes  to
  infinity. Thus Proposition  \ref{prop:mart-s} appears as a consequence
  of Proposition \ref{prop:mart-l}.
\end{rem}

Recall \reff{eq:def-g} and  \reff{eq:g-psil-qq}.
For $\theta\geq q$ and $q\in \Theta^{\psi,\lambda}$, we set:
\begin{equation}
   \label{eq:def-g-q}
g_q(r)=g_{(\psi_q, \psi_q(\eta))}(r)
\quad\text{and}\quad
g(r)=g_{(\psi,\lambda)}^{q,\theta}(r).
\end{equation}

\begin{proof}[Proof of  Proposition   \ref{prop:mart-l}]
We write $L_\theta$ for $L_\theta(\lambda)$. Let $q_0<q\leq \theta$.    By Theorem \ref{theo:growth}, we have:
\begin{equation}
   \label{eq:E-Lq-theta}
\E^{\psi,\lambda}\left[L_q|  \tau_\theta(\lambda)\right]
= L_\theta g(0) + L_\theta
g'(1) \E^{\psi,\lambda} \left[L_q \right].
\end{equation}
Thanks to Corollary \ref{cor:tau-law} and the branching property, we
have:
\begin{equation}
   \label{eq:N-Lq}
\E^{\psi,\lambda}\left[L_q\right]= g_q(0) + \E^{\psi,\lambda}
\left[L_q\right] g'_q(1). 
\end{equation}
This gives:
\[
\P^{\psi,\lambda}\left[L_q \right]= \frac{g_q(0)}{1- g'_q(1)}= \frac{\psi_q(\eta)}{\eta
  \psi'(q)} \cdot
\]
 Then use that:
\[
g(0)=\frac{\psi_q(\eta)}{\psi_\theta(\eta)},
\quad
g'(1) = \frac{\eta}{\psi_\theta(\eta)} \left( \psi'(\theta) -\psi'(q)
\right),
\]
and \reff{eq:E-Lq-theta} to get that:
\[
\E^{\psi,\lambda}\left[L_q| \tau_\theta(\lambda)\right]
= \frac{\psi_q(\eta)}{\psi'(q)} \frac{\psi'(\theta)}{\psi_\theta(\eta)} L_\theta.
\]
This gives the result. 
\end{proof}

\begin{remark}
A similar result  for the leaves process of  discrete time Galton-Watson
tree-valued process was proved in Corollary 3.4 of \cite{adh:pgwttvmp}
using a quantity similar to $(
1- g'_q(1))/g_q(0)$ which comes from  \reff{eq:N-Lq}.
\end{remark}

For $\theta>\theta_\lambda$, the function $g_\theta$ is convex
positive with $g_\theta(0)>0$ and $g_\theta(1)=1$. Hence, for $\zeta\in [0,1)$, the equation:
\[
x=g_\theta(x) + g_\theta(0) (\zeta-1)
\]
has   a   unique   solution   $x\in   [0,1]$,   which   we   denote   by
$h_\theta(\zeta)$.     By    construction    the     backward    process
$(L_\theta(\lambda),    \theta>\theta_\lambda)$    is    Markov    under
$\P^{\psi,\lambda}$.   The   next   Proposition   gives  its   one   and
two-dimensional marginals.
\begin{prop}
   \label{prop:Lq}
Let $\psi$ be a branching mechanism satisfying (H1-3). 
For $\theta\geq q> \theta_\lambda$ and $\zeta, z\in [0,1)$, we have:
\begin{equation}
   \label{eq:NzL}
\E^{\psi,\lambda}\left[\zeta^{L_\theta(\lambda)}\right]= h_\theta(\zeta),
\quad\text{and}\quad
\E^{\psi,\lambda}\left[\zeta^{L_\theta(\lambda)}z^{L_q(\lambda)}\right]=
h_\theta(\zeta w^{q,\theta}(z)),
\end{equation}
with $w^{q,\theta}(z)=  g(h_q(z))+g(0) (z- 1)$.
\end{prop}

\begin{proof} We write $L_\theta$ for $L_\theta(\lambda)$. 
  Conditioning on the  number of children of the  lowest branching point
  and using the branching property of the Galton-Watson trees
  $\tau_\theta(\lambda)$, we get: 
\[
   \E^{\psi,\lambda} \left[\zeta^{L_\theta}\right]
=g_{\theta}(0)\zeta+\sum_{k=1}^{\infty}
    \E^{\psi,\lambda}\left[\zeta^{L_\theta}\right]^k
\frac{g^{(k)}_{\theta}(0)}{k!}
=g_\theta\left(\E^{\psi,\lambda}\left[\zeta^{L_\theta}\right]\right)
+g_\theta(0) (\zeta-1). 
\]
This gives the first part of \reff{eq:NzL}.
Recall $\cg_{q,\theta}(\psi,\lambda)$ defined in Section
\ref{sec:subgrowth}.
Using again the branching
property, we have:
\[
 \E\left[z^{L(\cg_{q,\theta}(\psi,\lambda))}\right]
= g(0)z+ g(h_q(z)) - g(0).
\]
 Then,  by Theorem \ref{theo:growth}, we have:
\begin{align*}
\E^{\psi,\lambda}\left[\zeta^{L_\theta}z^{L_q}\right]
&=   \E^{\psi,\lambda}\left[\zeta^{L_\theta}z^{\sum_{x\in
      \Lf(\tau_\theta(\lambda))} 
    L(\cg^x_{q,\theta} (\psi,\lambda))} \right]\\
&=   \E^{\psi,\lambda}\left[\zeta^{L_\theta}( g(h_q(z)) + g(0)(z-1))^{L_\theta }
  \right]\\
&= h_\theta\Big(\zeta ( g(h_q(z)) + g(0)(z-1))\Big).
\end{align*}
This ends the proof.
\end{proof}

\begin{exple}
\label{ex:quad}
Assume     $\psi(u)=\beta u^2$,    with    $\beta>0$, so that
$\Theta=\R$ and $q_0=0$.  Let $\lambda>0$. We have
$\eta=\sqrt{\lambda/\beta}$, $\theta_\lambda=-\eta/2$ and
$\Theta^{\psi,\lambda}=[\theta_\lambda, +\infty )$. For
$\theta>\theta_\lambda$ and $\zeta \in [0,1)$,  we have:
\[
\E^{\psi,\lambda} \left[\zeta^{L_\theta(\lambda)}\right]=
\frac{\eta+\theta-\sqrt{\theta^2 \zeta+(1-\zeta)(\theta+\eta)^2}}{\eta}
\]
and     for $\theta\neq 0$ (see \ref{eq:Lq-m1} for $\theta_\lambda<\theta<0$):
\[
\E^{\psi,\lambda}\left[{L_\theta(\lambda)}\right]=
\frac{\eta+2\theta}{2|\theta|} \cdot
\]
For $\theta=\theta_\lambda$, we have $g_{\theta_\lambda}(0)=0$, and the
tree $\tau_{\theta_\lambda}(\lambda)$ is a Yule tree and has no leaf
(formally, we have
$\E^{\psi,\lambda}\left[\zeta^{L_{\theta_\lambda}(\lambda)}\right]=0$). 
\end{exple}

\section{Ascension time and tree at the ascension time}
\label{sec:ascension}
For convenience, we assume in this Section that $\psi$ is a critical
branching mechanism satisfying (H1-3). 

\subsection{Ascension process and Ascension time}
Let $\lambda>0$. Recall $\theta_{\lz}$ and $\Theta^{\psi,\lz}$ defined
in Section 6.2.  Define the ascension time on $\{M_\lambda\geq 1\}$:
\begin{equation}
   \label{eq:def-Al}
A_{\lz}=\inf\{\theta\in\Theta^{\psi,\lz}; \tau_{\theta}(\lz) \text{ is
  a compact tree.}\},
\end{equation}
where   we  make   the  convention   that  $\inf\emptyset=\theta_{\lz}$.
$\P^{\psi,\lambda}$-a.s.,    we   have    $A_{\lz}\leq0$.    Since,   by
construction,  $\tau_{\theta}(\lz)$ is  a compact  tree if  and  only if
$\ct_{\theta}$     is     a      compact     tree,     we     have     $
A_{\lz}=\inf\{\theta\in\Theta^{\psi,\lz}: \sigma_{\theta}<\infty\}$.

For  $\theta\in \Theta$, we  set $\bar{\theta}=\psi^{-1}(\psi(\theta))$,
so that  $\bar{\theta}$ is  the unique positive  number such
that:
\begin{equation}
   \label{eq:barthe}
\psi(\bar{\theta})=\psi(\theta).
\end{equation}
By  Theorem  6.5 of  \cite{ad:ctvmp}  and its  proof,  we  have for  all
$\theta\in \Theta^\psi$:
\begin{equation}
   \label{eq:dif-theta}
\bar{\theta}-\theta=\psi_{\theta}^{-1}(0).
\end{equation}
Recall   $g_\theta$    defined   in   \reff{eq:def-g-q}.    And   notice
$1-\frac{\bar{\theta}-\theta}{\eta}$  is  the  minimal solution  of  the
equation $r=g_\theta(r)$. Since
$\tau_{\theta}(\lz)$  is under $\P^{\psi,\lambda}$ a Galton-Watson tree
such the reproduction law has generating function $g_\theta$, we deduce that for   $\theta\in(\theta_{\lz},0)$:
\begin{equation}
   \label{eq:disA1}
\P^{\psi,\lambda}(A_{\lz}<\theta)=\P^{\psi,\lambda}(\tau_{\theta}(\lz)\text{
  is      compact})=1-\frac{\bar{\theta}-\theta}{\eta}\cdot
\end{equation}
Since $d\bar \theta/d\theta=\psi'(\theta)/\psi'(\bar \theta)$, we have
for   $\theta_{\lz}<\theta<0$:
\begin{equation}
   \label{eq:disA2}
\P^{\psi,\lambda}(A_{\lz}\in
d\theta)=\frac{1}{\eta}
\left(1-\frac{\psi'(\theta)}{\psi'(\bar{\theta})}\right)d\theta. 
\end{equation}
We give the distribution of the sub-tree at the ascension time. We set
$\cs_\theta(\lambda)=(\tau_{\theta+q} (\lambda), q\geq 0)$. 
Thanks to Corollary \ref{cor:tau-law} and Theorem \ref{theo:pruning}, for
$\theta\in \Theta^{\psi,\lambda}$, 
 $\cs_\theta(\lambda)$ under $\P^{\psi,\lambda}$
is 
distributed as $\cs_0(\psi_\theta(\eta))$ under $\P^{\psi_\theta,\psi_\theta(\eta)}$.

\begin{prop}
\label{prop:S-SL}
  Let      $\lambda>0$      and      $\eta=\psi^{-1}(\lambda)$.      For
  $\theta_{\lz}<\theta<0$ and any non-negative measurable function $F$, we have:
\[
\E^{\psi,\lz}[F(\cs_{A_{\lz}}(\lz))|A_{\lz}=\theta]
=\frac{\eta\psi'(\bar{\theta})}{\psi_{\theta}(\eta)}
\E^{\psi,\lz}\left[F(\cs_{\theta}(\lz))\, L_{\theta}(\lambda)\, 
\ind_{\{L_{\theta}(\lambda)<\infty\}}\right].
\]
\end{prop}

\begin{proof}
By considering $\E^{\psi,\lambda}\left[F(\cs_q)|\tau_q(\lambda)\right]$
instead of $F(\cs_q(\lambda))$, one can assume that $F$ is measurable defined on
$\T$. Assume $F(T)=0$ if $T$ is non compact. For $\theta_{\lz}<q<\theta<0$, we have:
\begin{equation}
   \label{eq:Asc-T}
\E^{\psi,\lz}[F(\tau_{\theta}(\lz))\ind_{\{A_{\lz}\geq q\}}]
=\E^{\psi,\lz}[F(\tau_{\theta}(\lz))\P^{\psi,\lz}(\tau_q(\lz) \text{ is
  non compact}|\tau_{\theta}(\lz))].
\end{equation}
We write $L_\theta$ for $L_\theta(\lambda)$. On
$\{\tau_\theta(\lambda)\text{ compact}\}$ that is 
$\{L_{\theta} <\infty\}$, we get that $\tau_q(\lambda)$ is
compact if and only if the trees grafted on $\tau_\theta(\lambda)$ to
get $\tau_q(\lambda)$, see Theorem
\ref{theo:growth}, are compact. Using \reff{eq:disA1},
\reff{eq:g-psil-qq} and notation \reff{eq:def-g-q}, we get on
$\{L_{\theta} <\infty\}$:
\begin{equation}
   \label{eq:Asc-T|}
\P^{\psi,\lz}(\tau_q(\lz) \text{ is  non compact}|\tau_{\theta}(\lz))
=1- g\left(1-\frac{\bar{q}-q}{\eta}\right)^{L_{\theta}}.
\end{equation}
A  simple  calculation  (recall  $g$   depends  on  $q$)  based  on  the
computation         of         \reff{eq:disA2}         yields         on
$\{L_{\theta}<\infty\}$:
\begin{align*}
   \frac{d}{dq}g
 \left(1-\frac{\bar{q}-q}{\eta}\right)^{L_{\theta}}_{|q=\theta}
&= L_\theta\,
g\left(1-\frac{\bar{q}-q}{\eta}\right)^{L_{\theta}-1}_{|q=\theta}
\frac{d}{dq}g \left(1-\frac{\bar{q}-q}{\eta}\right) _{|q=\theta}\\
&= L_\theta\left[
\frac{dg}{dq} \left(1-\frac{\bar{q}-q}{\eta}\right) 
- g' \left(1-\frac{\bar{q}-q}{\eta}\right) \frac{1}{\eta}
\left(1-\frac{\psi'(q)}{\psi'(\bar{q})}\right)
\right]_{|q=\theta} \\
&= L_{\theta}\, 
 \frac{\psi'(\bar{\theta})}{\psi_{\theta}(\eta)}\cdot
\end{align*} 
Then by \reff{eq:Asc-T} and \reff{eq:Asc-T|} and thanks to the
regularity of $g$ and $\bar q$ in $q$, we have:
\begin{align*}
  \frac{\E^{\psi,\lz}[F(\tau_{\theta}(\lz)),A_{\lz}\in d\theta]}{d\theta}
 &=-\frac{d}{dq}\E^{\psi,\lz}[F(\tau_{\theta}(\lz))
\ind _{\{A_{\lz}\geq q\}}]_{|q=\theta}\\
& =\E^{\psi,\lz}\left[F(\tau_{\theta}(\lz))L_{\theta}\, 
 \frac{\psi'(\bar{\theta})}{\psi_{\theta}(\eta)}\ind
 _{\{L_{\theta} <\infty\}}\right].  
\end{align*}
Meanwhile, by Proposition \ref{prop:Lq}, we have:
\begin{equation}
   \label{eq:Lq-m1}
\E^{\psi,\lz}[L_{\theta} \ind _{\{L_{\theta}<\infty\}}]
=\lim_{\zeta\rar1-}\frac{\partial}{\partial \zeta}h_{\theta}(\zeta)
=\frac{g_\theta(\eta)(0)}
 {1-g'_{\theta}(h_{\theta}(1-))}
 =\frac{\psi_{\theta}(\eta)}{\eta\psi'(\bar{\theta})},
\end{equation}
where we use the fact that
$h_{\theta}(1-)=1-\frac{\bar{\theta}-\theta}{\eta}$ which  is  the
minimal solution  of  the equation $r=g_\theta(r)$.
Thus, we get: 
\begin{align*}
\E^{\psi,\lz}[F(\tau_{\theta}(\lz))|A_{\lz}\in d\theta]
&=\frac{\E^{\psi,\lz}[F(\tau_{\theta}(\lz)),A_{\lz}\in
  d\theta]}{\P^{\psi,\lz}(A_{\lz}\in d\theta)}\\
&=\frac{\E^{\psi,\lz}[F(\tau_{\theta}(\lz)) L_\theta
  \ind_{\{L_\theta<+\infty
    \}}]}{\E^{\psi,\lz}[L_{\theta}  \ind_{\{L_\theta<+\infty
    \}}]}\\
&= \frac{\eta\psi'(\bar{\theta})}{\psi_{\theta}(\eta)}\; 
\E^{\psi,\lz}\left[F(\tau_{\theta}(\lz))L_{\theta}\ind_{\{L_{\theta}<\infty\}}
  \right].   
\end{align*}
This ends  the proof. 
\end{proof}

We give an immediate Corollary. 

\begin{cor}
   \label{cor:S-l-S-L}
  Let      $\lambda>0$      and      $\eta=\psi^{-1}(\lambda)$.      For
  $\theta_{\lz}<\theta<0$ and any non-negative measurable function $F$,
  we have, with $\eta_\theta=\eta-\bar \theta +\theta$:
\[
\E^{\psi,\lz}[F(\cs_{A_{\lz}}(\lz))|A_{\lz}=\theta]
=\frac{\eta_ \theta \psi'(\bar{\theta})}{\psi_{\bar
    \theta}(\eta_\theta)}
\P^{\psi,\lambda}\left[F(\cs_{\bar{\theta}}(\psi(\eta_\theta)))
L_{\bar{\theta}}(\psi(\eta_\theta))\right].
\]
\end{cor}
\begin{proof}
   Then similarly to Proposition 4.6 of \cite{adh:pgwttvmp}, we have:
\begin{align*}
\E^{\psi,\lambda}[F(\cs_{A_\lambda}(\lz))|A_{\lz}= \theta]
&=\frac{\eta\psi'(\bar{\theta})}{\psi_{\theta}(\eta)}
\E^{\psi,\lz}\left[F(\cs_{\theta}(\lz))L_{\theta}(\lambda)
\ind_{\{L_{\theta}(\lambda)<\infty\}}\right]\\
&=
\frac{\psi'(\bar{\theta})}{\psi_{\theta}(\eta)}
\N^{\psi}\left[F(\cs_{\theta}(\lz))L_{\theta}(\lambda)
\ind_{\{1\leq L_{\theta}(\lambda)<\infty\}}\right]\\
&=
\frac{\psi'(\bar{\theta})}{\psi_{\theta}(\eta)}
\N^{\psi_{\theta}}\left[F(\cs_{0}(\psi_{\theta}(\eta)))
L_{0}(\psi_{\theta}(\eta))
 \ind_{\{1\leq L_{0}(\psi_{\theta}(\eta)))<\infty\}}\right]\\
&=
\frac{\psi'(\bar{\theta})}{\psi_{\theta}(\eta)}
\N^{\psi_{\bar{\theta}}}\left[F(\cs_{0}(\psi_{\theta}(\eta)))
L_{0}(\psi_{\theta}(\eta))
\ind_{\{L_{0}(\psi_{\theta}(\eta))\geq 1\}}\right]\\
&=
\frac{\psi'(\bar{\theta})}{\psi_{\theta}(\eta)}
\N^{\psi}\left[F(\cs_{\bar{\theta}}(\psi(\eta-\bar{\theta}+\theta)))
L_{\bar{\theta}}(\psi(\eta-\bar{\theta}+\theta))
\ind_{\{L_{\bar{\theta}}(\psi(\eta-\bar{\theta}+\theta))\geq 1\}}\right]\\
&=
\frac{(\eta-\bar \theta+\theta) \psi'(\bar{\theta})}{\psi_{\bar
    \theta}(\eta-\bar \theta + \theta)}
\P^{\psi,\lambda}\left[F(\cs_{\bar{\theta}}(\psi(\eta-\bar{\theta}+\theta)))
L_{\bar{\theta}}(\psi(\eta-\bar{\theta}+\theta))\right],
\end{align*}
where we used Proposition \ref{prop:S-SL} for the first equality;
definition \reff{eq:def-P-psi-l} of $\P^{\psi,\lambda}$ and 
$\{M_\lambda\geq 1\}= \{L_\theta(\lambda)\geq 1\}$ for the second;
Proposition \ref{prop:marginal} for the third; Girsanov transformation
\reff{eq:GN} and  $\psi_\theta(\bar \theta-\theta)=0$ as well as the
fact that the number of the leaves are finite under $\N^{\psi_{\bar
    \theta}}$ as $\psi_{\bar \theta}$ is sub-critical for the fourth;
Proposition \ref{prop:marginal} as well as the equality 
$\psi_{\bar{\theta}}(\eta-{\bar{\theta}}+\theta)=\psi_{\theta}(\eta)$
for the fifth and sixth equalities.
\end{proof}

\subsection{An infinite CRT and its pruning}
\label{sec:infinite-CRT}
An  infinite  CRT was  constructed  in  \cite{ad:ctvmp}  which, because
of (H2)  is the L\'evy CRT conditioned to have infinite height. Notice
that since $\psi$ is critical the event of infinite height is of measure
zero. Before recalling
its construction, we stress that under $\P_r^\psi$, the root $\emptyset$
belongs  to  $\Br_\infty  $  and  has  mass  $\Delta_\emptyset=r$.  We
identify the half real line $[0,+\infty )$ with a real tree denoted by
$\llbracket 0,\infty \llbracket$ with the null mass measure. We denote
by $dx$ the length measure on $\llbracket 0,\infty \llbracket$. 
 Let
$\sum_{i\in I^*} \dz_{(x^{*}_i, T^{*,i})}$ be a Poisson point measure    on
$\llbracket 0,\infty \llbracket \times    {\T}$    with     intensity
$dx \, \bN^\psi[d\ct]$, with $\bN^\psi[d\ct]$ defined in \reff{eq:def-bN}.
The infinite CRT from \cite{ad:ctvmp}  is defined as:
\begin{equation}
   \label{eq:T*}
\ct^*=\lb\emptyset,\infty\lb\circledast_{i\in I^*}(T^{*,i}, x_i^{*}).
\end{equation}
We denote by $\P^{*,\psi}(d\ct^*)$ the distribution of $\ct^*$. 
Following  \cite{ad:ctvmp}  and  similarly  to the  setting  in  Section
\ref{sec:pruning},   we    consider   on   $\ct^*$    a   mark   process
$M^{\ct^*}(d\theta,dy)$ which is a  Poisson point measure on $\R_+\times
\ct^*$ with intensity:
\[
\ind_{[\emptyset,+\infty)}(\theta)d\theta\left(2\beta \ell^{\ct^*}(dy)+\sum_{i\in
   I^* }\sum_{x\in \Br_{\infty}(T^{*,i  })}\Delta_x\delta_x(dy)\right), 
\]
with the identification of $x^{*}_i$ as the root of $\ct^{*,i}$. In
particular nodes in $\lb\emptyset,\infty\lb$ with infinite degree will
be charged by $M^{\ct^*}$. 
For every $x\in\ct^*$, we set:
\[
\theta^*(x)=\inf\{\theta>0,\ M^{\ct^*}([0,\theta]\times \lb\emptyset,x\rb)>0\}.
\]
Then we  define the  pruned tree at  time $q$  as $\ct_q^*=\{x\in\ct^*,\
\theta^*(x)\ge q\}$  with the induced metric, root  $\emptyset$ and mass
measure the restriction to $\ct_q^*$ of the mass measure $\bm^{\ct^*}$.

Given $\ct^*$, let $\cp^{*}(dtdx)=\sum_{j\in J^*}\dz_{(t_j^*, y_j^*)}$
be a Poisson point  measure on $[0,\infty)\times\ct^*$ with intensity
$dt\, {\bf m}^{\ct^*}(dx)$. For
$\theta\geq 0$ and $\lz>0$, define the  pruned  sub-tree
$\tau_\theta^*(\lambda)$  containing the  root  and all  the ancestors  in
$\ct_\theta^*$ of the marked leaves of $\ct^*$:
\begin{equation}
   \label{subtree*}
\tau_{0}^*(\lambda)=\bigcup_{j\in J^*, t_j^*\leq  \lambda}
\llb\emptyset,  y_j^*\rrb 
\quad\text{and}\quad 
\tau_{\theta}^*(\lambda)=\tau_{0}^*(\lambda)\bigcap \ct_\theta^*. 
\end{equation}
We define $
\tau_{\theta}^*(0)=\bigcap_{\lz>0}\tau_{\theta}^*(\lambda)$,
and notice that $\tau_{\theta}^*(0)=\lb\emptyset,\infty\lb$ and that it
has no leaf. 
Similarly to \reff{eq:mass-tql}, we define the mass measure of
$\tau_{\theta}^*(\lambda)$ by:
\begin{equation}
   \label{eq:mass-tql*}
\bm^{\tau_\theta^*(\lambda)}=\inv{\psi_\theta(\eta)} \sum_{x\in
  \Lf (\tau_\theta^*(\lambda))} \delta_x,
\end{equation}
with $\eta=\psi^{-1}(\lambda)$ and the convention the mass measure is
zero if $\lambda=0$.  

We have a similar convergence result as  Theorem \ref{theo:cv-dGHPc}. 

\begin{theo}
   \label{theo:cv-dGHP*}
For all $\theta\geq 0$, we have $\P^{*,\psi}$-a.s.:
\[
\lim_{\lambda\rightarrow+\infty } d_{\text{GHP}}(\ct^*_\theta,
\tau_\theta^*(\lambda))=0.
\]
\end{theo}

\begin{proof}
   According to \cite{ad:ctvmp}, there exists a family of random continuous
   functions $(H^{(a)}, a>0)$ with compact support such that: $H^{(a)}$
     takes values in $[0,a]$;  for all
   $0<b<a$ and $t\geq 0$, we have:
\[
H^{(b)}(t)=H^{(a)}(C^{-1}_{b,a}(t))
\quad\text{with}\quad
C_{b,a}(s)=\int_0^s \ind_{\{H^{(a)}(r)\leq b\}}\, dr;
\]
and $\left(\left(\ct_\theta^*\right)^{(a)}, a>0\right)$ under $\P^{*,\psi}$ is
distributed as $\left(\ct^{H^{(a)}}, a>0\right)$. Following the proof of
Lemma \ref{lem:cv-dGHPc}, we get that for all $a>0$, $\P^{*,\psi}$ a.s.
\[
\lim_{\lambda\rightarrow+\infty
}d_{\text{GHP}}^c\left(\left(\ct_\theta^*\right)^{(a)} ,
  \left(\tau_\theta^*(\lambda)\right)^{(a)} \right) =0.
\]
This and the definition of $d_{\text{GHP}}$ gives the result. 
\end{proof}


\begin{rem}
  \label{rem:recons}   Similarly   to  Theorem   \ref{theo:decomp-tree},
  according  to the  argument  in \cite{cw:ggwt},  we could  reconstruct
  $\ct^*$   from   $\tau_0^*(\lambda)$.    Recall   \reff{eq:def-Gamma}.
  Conditionally on $\tau_0^*(\lambda)$, $\ct_0^*$ is distributed as:
\[
\tilde \tau_0^*(\lambda)
\circledast_{i\in I}(\ct_i^*, x_i^*)
\circledast_{x\in \Br(\tilde\tau_0^*(\lambda))} (\ct_x^*, x),
\]
with:
\begin{itemize}
   \item $\tilde \tau_0^*(\lambda)$ as $\tau_0^*(\lambda)$ but with 0 as mass
measure,
\item 
$\sum_{i\in I} \delta_{(x_i^*, \ct_i^*)}$ is a random Poisson point measure
on $\tilde\tau_0^*(\lambda)\times \T$ with intensity given by
$\ell^{\tilde \tau_0^*(\lambda)}(dx)\, \bN^{\psi_\eta}[d\ct]$,
\item conditionally on 
$\sum_{i\in I} \delta_{(x_i^*, \ct_i^*)}$, the trees $\left(\ct^*_x, x\in
  \Br(\tilde \tau_0^*(\lambda) )\right)$ are independent
  with $\ct_x^*$ is distributed as:
\[
\int \Gamma^\psi_{\kappa(x),\lambda} (dr) \;
  \P_r^{\psi_\eta}[d\ct].
\]
 \end{itemize}

\end{rem}


\subsection{Distribution of the sub-tree of the infinite CRT}
\label{sec:dist-infinite-CRT}
Recall that $\tilde \tau_0(\lambda)$ is under $\P^{\psi,\lambda}$ a
Galton-Watson tree with distribution $\rP^{\psi,\lambda}$. We shall now
describe the distribution of  $\tilde \tau_0^*(\lambda)$  under
$\P^{\psi}$, which can be seen as a  Galton-Watson  tree with distribution $\rP^{\psi,\lambda}$
conditionally on the non extinction event.

Let $K$ be an integer-valued random variable with generating function
$g_{(\psi,\lambda)}$  defined by \reff{eq:def-g}. 
Since $\psi$ is critical, we have $g'_{(\psi,\lambda)}(1)=1$, which
implies that $g'_{(\psi,\lambda)}$ itself is the generating function of
a integer-valued random variable, say $K^*$. Since $g'_{(\psi,\lambda)}(0)=0$,
$K^*$ is a.s. positive. Notice that  the distribution of $K^*+1$
is the size-biased distribution of $K$. Let $(\tau^{k,*},
k\in \N^*)$ be independent random trees distributed as $\tau_0(\lambda)$
under $\P^{\psi,\lambda}$ (that is with distribution
$\rP^{\psi,\lambda}$ and mass measure given by  \reff{eq:mass-tql})
independent of $K^*$. We set:
\[
{\G}^*= \emptyset \circledast_{1\leq k\leq
  K^*}(\tau^{k,*},\emptyset) . 
\]

\begin{theo}\label{Propsub*} Let  $\lz>0$ and
  $\eta=\psi^{-1}(\lambda)$. Under $\P^{*,\psi}$,  $\tau_0^*(\lz)$ is
a rooted real tree distributed as:
\[
\lb \emptyset,\infty\lb
\circledast_{i\in I_0^*}(\G^{*,i},x_i^*),
\]
where $\sum_{i\in I_0^*}\dz_{x_i^*}$  is a  Poisson point measure on $\lb
\emptyset,\infty\lb$ with intensity $\psi'(\eta)dx$ and conditionally on
this Poisson point measure, the real trees $(\G^{*,i}, i\in I_0^*)$ are
independent  and distributed as $\G^*$. 
\end{theo}

\begin{proof}
By construction, thanks to \reff{eq:T*}, we have:
\[
\tau_0^*(\lz)= \lb\emptyset,\infty\lb\circledast_{i\in
  I^*}(\tau^{*,i}(\lambda), x^*_i),
\]
with $\tau^{*,i}(\lambda)=\bigcup_{j\in J^*, t_j^*\leq  \lambda,
  x^*_i\preccurlyeq y^*_j}
\llb x_i^*,  y_j^*\rrb $ distributed as $\tau_0(\lambda)$ under
$\bN^\psi[d\ct]$. The marked Poisson point measure $\sum_{i\in I^*} 
\ind_{\tau^{*,i}(\lambda)\neq \emptyset\}}\dz_{x_i^*}$ is a Poisson
point measure on  
$\lb\emptyset,\infty\lb$ with intensity $\bN^\psi[M_\lambda\geq 1]\, dx
=\psi'(\eta)\, dx$. 

Let $I^*_0=\{i\in I^*; \tau^{*,i}(\lambda)\neq
\emptyset\}$. The  sub-trees $(\tau^{*,i}(\lambda), i\in I^*_0)$ are
independent and distributed as $\tau_0(\lambda)$ under
$\bN^\psi[\, \cdot\, | M_\lambda\geq 1]$. Let $N_\emptyset$ be the degree
of the root of $\tau_0(\lambda)$. The theorem will be proved once we
check that $N_\emptyset$  under $\bN^\psi[\, \cdot\, | M_\lambda\geq 1]$  is
distributed as $K^*$. 
Following the proof of Lemma
\ref{lem:cn-tl}, we set $h^*(u) =  \bN^\psi\left[u^{N_\emptyset}
  \ind_{\{N_\emptyset\geq 1\}} \right]$, and we have for $u\in [0,1]$:
\begin{align*}
 h^*(u)
&= 2 \beta \N^\psi[M_\lambda\geq 1] u +
\int_{(0,+\infty )} r\pi(dr) \,
\E^\psi_r\left[u^{N_\emptyset}\ind_{\{N_\emptyset\geq 1\}}\right]\\
&= 2 \beta \eta u + \int_{(0,+\infty )} r\pi(dr) \, \left(\expp {-r\eta
    (1-u)}-  \expp{-r\eta}\right). 
\end{align*}
Elementary computations yield $g_{(\psi,\lambda)}'(u)=h^*(u)/h^*(1)$. Thus
$N_\emptyset$ under $\bN^\psi[\, \cdot\, | M_\lambda\geq 1]$ is
distributed as $K^*$. This ends the proof. 
\end{proof}

We give  a similar representation  formula for $\tau^*_\theta(\lambda)$.
Let $K^*_\theta$  be an  integer-valued random variable  with generating
function $g_\theta'/g_\theta'(1)$, see definitions \reff{eq:def-g-q} and
\reff{eq:def-g}.     Since     $g'_\theta(0)=0$,     $K^*_\theta$     is
a.s.  positive. Notice that  the distribution  of $K_\theta^*+1$  is the
size-biased   distribution  of   $K_\theta$  with   generating  function
$g_\theta$.  Let $(\tau_\theta^{k,*},  k\in \N^*)$ be independent random
trees  distributed as  $\tau_\theta(\lambda)$  under $\P^{\psi,\lambda}$
(that  is  with  distribution $\rP^{\psi_\theta,\psi_\theta(\eta)}$  and
mass    measure   given    by    \reff{eq:mass-tql})   independent    of
$K^*_\theta$. We set:
\[
\G_\theta^*= \emptyset \circledast_{1\leq k\leq
  K^*_\theta}(\tau_\theta^{k,*},\emptyset) . 
\]

\begin{theo}
\label{theo:t*_q}
 Let  $\lz>0$ and
  $\eta=\psi^{-1}(\lambda)$. For  $\theta>0$, 
under $\P^{*,\psi}$,  $\tau_\theta^*(\lz)$ is
a rooted real tree distributed as:
\[
\lb \emptyset,E_\theta\rb
\circledast_{i\in I_\theta^*}(\G_\theta^{*,i},x_i^*),
\]
where 
\begin{itemize}
\item $\lb \emptyset,E_\theta\rb$ is a  real tree rooted at $\emptyset$ with no
  branching point  and zero mass measure and  such that $d(\emptyset,E_\theta)$
  is an exponential random variable with parameter $\psi'_\theta(0)$,
   \item $\sum_{i\in I_\theta^*}\dz_{x_i^*}$  is an independent 
     Poisson point measure on $\lb 
\emptyset, E_\theta\rb$ with intensity $[\psi'_\theta(\eta)-
\psi'_\theta(0)]\, dx$,
\item conditionally on $E_\theta$ and $\sum_{i\in I_\theta^*}\dz_{x_i^*}$, the
  real trees $(\G^{*,i}, i\in I_0^*)$ are 
independent  and distributed as $\G^*_\theta$. 
\end{itemize}
\end{theo}

\begin{proof}
   Recall notations of the proof of Theorem \ref{Propsub*}. 
The distribution of $d(\emptyset,E_\theta)$ is given in
   \cite{ad:ctvmp}. By construction, thanks to \reff{eq:T*}, we have:
\[
\tau_0^*(\lz)= \lb\emptyset,E_\theta\rb\circledast_{i\in
  I^*}(\tau^{*,i}_\theta(\lambda), x^*_i),
\]
with               $\tau^{*,i}_\theta(\lambda)=\tau^{*,i}(\lambda)\bigcap
\ct^{*}_\theta$. Let $N_{\emptyset,\theta}$ (resp. $N_\emptyset'$) be the
degree       of      the       root       of      $\tau_\theta(\lambda)$
(resp.         $\tau_0(\psi_\theta(\eta))$).         Notice         that
$\tau^{*,i}_\theta(\lambda)$  is  distributed as  $\tau_\theta(\lambda)$
under    $\bN^\psi[d\ct,   \,   N_\emptyset\geq    1]$   that    is   as
$\tau_0(\psi_\theta(\eta))$ under $\bN^{\psi_\theta}[d\ct]$. The rate at
which sub-trees are grafted on the spine 
$\lb\emptyset,E_\theta\rb$ is given by:
\[
\bN^{\psi_\theta}\left[N_\emptyset'\geq 1\right]= \psi_\theta'(\eta)
-\psi'_\theta(0).
\]
Then to end
the proof, it is
enough to check that  $N_\emptyset'$ under $\bN^{\psi_\theta}[\, \cdot\,
| \, N_\emptyset\geq 1]$ is distributed as $K^*_\theta$. Elementary
computations give:
\[
h^*_\theta(u)=\bN^{\psi_\theta}\left[u^{N_\emptyset'}
  \ind_{\{N_\emptyset'\geq 1\}}\right]
= \psi_\theta(\eta) -\psi'_\theta(\eta(1-u)),
\]
so  that  $h^*_\theta(u)/h^*_\theta(1)=g'_\theta(u)/g'_\theta(1)$.  Thus,
$N_\emptyset'$ under $\bN^{\psi_\theta}[\,  \cdot\, | \, N_\emptyset\geq
1]$ is distributed as $K^*_\theta$.
\end{proof}

We also provide a recursive distribution of the tree
$\tau^*_\theta(\lambda)$. Let
$a_\theta(\lambda)=\psi_\theta'(0)/\psi'_\theta(\eta)=1-g'_\theta(1)$.

\begin{cor}
\label{cor:t*_q} 
Let  $\lz>0$ and
  $\eta=\psi^{-1}(\lambda)$. For  $\theta>0$, under $\P^{*,\psi}$, 
 $\tau_{\theta}^*(\lz)$ is
a rooted real tree distributed as $\lb
\emptyset,E_\theta(\lambda)\rb$ with probability $a_\theta(\lambda)$
 and with probability $1-a_\theta(\lambda)$ as:
\[
\lb
\emptyset,E_\theta(\lambda)\rb
\circledast_{0\leq i\leq 1}(\G^{*,i}_\theta,E_\theta(\lambda)),
\]
where
\begin{itemize}
\item $\lb \emptyset,E_\theta(\lambda)\rb$ is  a real tree rooted at $\emptyset$
  with  no  branching  point  and   zero  mass  measure  and  such  that
  $d(\emptyset,E_\theta(\lambda))$  is  an   exponential  random  variable  with
  parameter $\psi'_{\theta}(\eta)$,
   \item conditionally on $E_\lambda(\theta)$, $\G^{*,0}_\theta$ and
     $\G^{* ,1}$ are independent and distributed respectively as
     $\G^*_\theta$ and  $\tau^*_\theta(\lambda)$. 
        \end{itemize}
\end{cor}

Notice that the number of children of $E_\theta(\lambda)$ has generating
function $1-g'_\theta(1)+ u g'_\theta(u)$. 

\begin{proof}
   This is a direct consequence of Theorem \ref{theo:t*_q}, when
   considering the decomposition of $\tau_{\theta}^*(\lz)$ with respect
   to the lowest branching point and using the branching
   property. Notice that there is no such  branching point (and then
   $\tau^*_\theta(\lambda)$ is reduce to a spine) if the point measure
   $\sum_{i\in I^*_\theta} \delta_{x^*_i} $ defined in Theorem
   \ref{theo:t*_q} is zero. This happens with probability
   $a_\theta(\lambda)$. 
\end{proof}

\begin{rem}
   \label{rem:pruning*}
Notice that $\tau_{\theta}^*(\lz)$ could be obtained from
$\tau_0^*(\lz)$ by a similar pruning procedure as the one defined in Section 6.1.
\end{rem}

\subsection{Sub-tree process from the ascension time}
We start with a remark on size-biased discrete Galton-Watson trees. 
\begin{rem}
   \label{rem:GW} 
Let $\cg$ be a discrete sub-critical Galton-Watson tree starting with one
root and with $g$ as
generating function of the reproduction law. Let $L$ be the number of
leaves of $T$. We have $\E[L]=g(0)/[1-g'(1)]$. Let $\cg^*$ be distributed
as the size-biased distribution of $T$ with respect to $L$, that is for
any non-negative measurable function:
\[
\E[F(\cg^*)]=\frac{\E[LF(\cg)]}{\E[L]}\cdot
\]
Let $N_\emptyset(t)$  be the number  of children of  the root of  a tree
$t$.  For  example $N_\emptyset(\cg)$  has  generating  function $g$.  The
following result can be proved  inductively by decomposing the tree with
respect to the children of the root.

 The distribution of $\cg^*$ is characterized as follows.
$N_\emptyset(\cg^*)$ has generating function $u\rightarrow 
1-g'(1)+ug'(u)$. If $N_\emptyset(\cg^*)\geq 1$, then label from 1 to
$N_\emptyset(\cg^*)$ the children of the root and by $ \cg_i$ the sub-tree
attached to children $i$. Then $(\cg_1, \ldots, \cg_{N_\emptyset(\cg^*)})$ are independent trees; they are  distributed as $\cg$
but for $\cg_I$, for some random index $I$ uniform on $\{1, \ldots,
N_\emptyset(\cg^*)\}$, which is distributed as $\cg^*$. 
\end{rem}

We denote
$\cs_\theta^*(\lambda)=(\tau_{\theta+q}^*(\lambda), q\geq 0)$. 

\begin{proposition}
\label{Peqdis}
For $\theta>0$, $\lz>0$ and non-negative functionals $F$, we have:
\[
\frac{\eta\psi'(\theta)}{\psi_{\theta}(\eta)}\E^{\psi,\lambda}
[F(\cs_{\theta}(\lz))L_{\theta}(\lambda) ]
=\E^{*,\psi}[F(\cs^{*}_{\theta}(\lz))].
\]
\end{proposition}
\begin{proof} By considering $\E^{\psi,\lambda}\left[F(\cs_\theta)|\tau_\theta(\lambda)\right]$
instead of $F(\cs_\theta(\lambda))$ and $\E^{*,\psi}\left
  [F(\cs^{*}_{\theta}(\lz)) |\tau_\theta^*(\lambda)\right] $  instead of
$F(\cs^{*}_{\theta}(\lz)) $,  one can assume that $F$ is measurable defined on
$\T$.   Since   the   life   times   of   all   individuals   in
  $\tau_{\theta}(\lz)$   and   $\tau_{\theta}^*(\lz)$   have  the   same
  distribution, we only need to  consider the distribution of the number
  of  offsprings.   This is  equivalent  to  consider the  corresponding
  discrete  (or  size-biased) Galton-Watson  tree.  Then  the result  follows
  from Remark \ref{rem:GW}. 
\end{proof}

Recall that the function $\theta\mapsto \bar \theta$  is defined by \reff{eq:barthe}. If
$\theta_\lambda\in \Theta^\psi$, then we deduce from \reff{eq:dif-theta}
that:
\[
 \bar{\theta}_{\lz}-\theta_{\lz}=\eta. 
\]
In particular the function $f$ defined by:
\[
f_\lambda(r)=\frac{1}{\eta}
\left(1-\frac{\psi'(r)}{\psi'(\bar{r})}\right)\ind_{\{r\in(\theta_{\lz},0)\}}
\]
is a probability density. the corresponding cumulative distribution is
$F_\lambda$ defined on $[\theta_\lambda,0)$ by:
\[
F_{\lz}(r)=1- \frac{\bar r -r}{\eta}=\P^{\psi,\lz}(A_{\lz}<r).
\]

\begin{proposition}\label{disequ}
  Let   $\lambda>0$    and   $\eta=\psi^{-1}(\lambda)$.    Assume   that
  $\theta_\lambda\in  \Theta^\psi$. Under  $\P^{*,\psi}$, let  $U$  be a
  random   variable  with   density  $f_\lambda$   and   independent  of
  $\cs^*_0(\lambda)$.        Then       $\cs_{A_{\lz}}(\lz)$       under
  $\P^{\psi,\lambda}$      has     the     same      distribution     as
  $\cs^*_{\bar{U}}\left(\psi\left(\eta  F_{\lz}(U)\right)\right)$  under
  $\P^{*,\psi}$.
\end{proposition}

\begin{proof}
Using Corollary \ref{cor:S-l-S-L}, with $\eta_\theta=\eta-\bar \theta
+\theta$, we get:
\begin{align*}
\E^{\psi,\lz}[F(\cs_{A_{\lz}}(\lz))|A_{\lz}=\theta]
&=\frac{\eta_ \theta \psi'(\bar{\theta})}{\psi_{\bar
    \theta}(\eta_\theta)}
\E^{\psi,\lambda}\left[F(\cs_{\bar{\theta}}(\psi(\eta_\theta)))
L_{\bar{\theta}}(\psi(\eta_\theta))\right]\\
&=\E^{*,\psi}[F(\cs^{*}_{\bar \theta}(\psi(\eta_\theta)))]   .
\end{align*}
Using \reff{eq:disA2} and $\eta_\theta=\eta F_\lambda(\theta)$, we get:
\begin{align*}
&=  \int_{\theta_\lambda}^0
\E^{\psi,\lz}[F(\cs_{A_{\lz}}(\lz))|A_{\lz}=\theta]\, \P^{\psi,\lambda}
(A_\lambda\in d\theta)\\
&= \int_{\theta_{\lz}}^0  \E^{*,\psi}
[F(\cs^*_{\bar \theta}(\psi(\eta_\theta)))]
 f_\lambda (\theta) \, d\theta\\
&=
  \int_{\theta_{\lz}}^0
 \E^{*,\psi}[F(\cs^*_{\bar{\theta}}(\psi(\eta F_{\lz}(\theta))))]
 f_\lambda (\theta) \, d\theta\\
&=
 \E^{*\lambda} [F(\cs^*_{\bar{U}}(\psi(\eta F_{\lz}(U))))].
\end{align*}
\end{proof}



\end{document}